\RequirePackage{etoolbox}
\csdef{input@path}{{style/}{graphics/}}
\documentclass[seceqn,MSNbibl,number,citesort,dvips]{arxbj}
\usepackage{mathbh,upgreek}

\aid{0}
\volume{21}
\issue{1}
\pubyear{2015}
\firstpage{209}
\lastpage{241}
\doi{10.3150/13-BEJ565} 

\makeatletter

\newcommand{\rright}{\right}
\newcommand{\lleft}{\left}

\newtheorem{prp}{Proposition}
\newtheorem{lem}{Lemma}
\newtheorem{cor}{Corollary}
\newremark{Remark}{Remark}
\newcommand{\Var}{\operatorname{Var}}
\newcommand{\Cov}{\operatorname{Cov}}
\renewcommand{\P}{\mathrm{\mathbb{P}}}
\newcommand{\SDP}{\operatorname{SDP}}
\newcommand{\med}{\operatorname{med}}

\newcommand{\trace}{\operatorname{tr}}

\newcommand{\eps}{\varepsilon}

\newcommand{\cN}{\mathcal{N}}
\newcommand{\cM}{\mathcal{M}}
\newcommand{\cC}{\mathcal{C}}

\newcommand{\bbR}{\mathbb{R}}

\def\eps{\varepsilon}

\newcommand{\field}[1]{\mathbb{#1}}
\newcommand{\R}{\field{R}}
\newcommand{\E}{\field{E}}
\newcommand{\EXP}{\E}
\newcommand{\PROB}{\field{P}}
\newcommand{\IND}[1]{\mathbh{1}_{\{ #1 \}}}
\newcommand{\defeq}{\stackrel{\mathrm{def}}{=}}
\newcommand{\C}{\mathcal{C}}

\newcommand{\Rho}{\rho_{\mathrm{ave}}}
\newcommand{\binom}[2]{\pmatrix{#1\cr#2}}
\newcommand{\eqref}[1]{(\ref{#1})}
\renewcommand{\pi}{\uppi}
\renewcommand{\emptyset}{\varnothing}
\newcommand{\oo}{\mathrm{o}}
\newcommand{\OO}{\mathrm{O}}
\makeatother

\begin{document}
\begin{frontmatter}

\title{Detecting positive correlations in a multivariate sample}
\runtitle{Detecting positive correlations in a multivariate sample}

\begin{aug}
\author[1]{\inits{E.}\fnms{Ery} \snm{Arias-Castro}\thanksref{1}\ead[label=e1]{eariasca@math.ucsd.edu}},
\author[2]{\inits{S.}\fnms{S{\'e}bastien} \snm{Bubeck}\corref{}\thanksref{2}\ead[label=e2]{sbubeck@princeton.edu}} \and
\author[3]{\inits{G.}\fnms{G{\'a}bor} \snm{Lugosi}\thanksref{3}\ead[label=e3]{gabor.lugosi@upf.edu}}
\address[1]{Department of Mathematics,
University of California, San Diego, La Jolla, CA 92093, USA.\\ \printead{e1}}
\address[2]{Department of Operations Research and Financial Engineering,
Princeton University, Princeton, NJ 08544, USA. \printead{e2}}
\address[3]{Department of Economics, Pompeu Fabra University, 08005
Barcelona, Spain.\\ \printead{e3}}
\end{aug}

\received{\smonth{7} \syear{2012}}
\revised{\smonth{3} \syear{2013}}

%
\begin{abstract}
We consider the problem of testing whether a correlation matrix of a
multivariate normal population is the identity matrix. We focus on
sparse classes of alternatives where only a few entries are nonzero
and, in fact, positive. We derive a general lower bound applicable to
various classes and study the performance of some near-optimal tests.
We pay special attention to computational feasibility
and construct near-optimal tests that can be computed efficiently.
Finally, we apply our results to prove new lower bounds for the clique number
of high-dimensional random geometric graphs.
\end{abstract}

%

\end{frontmatter}

\section{Introduction}

In multivariate statistics, inference about a covariance (i.e.,
dispersion) matrix aims at answering questions of dependencies between
the variables. This is strictly true when the variables are jointly
Gaussian, which is the classical assumption. A basic question is
whether the variables are dependent at all. Concretely, consider a
simple setting where the components of a random vector
are jointly normal, each with zero
mean and unit variance. Then the variables are independent if and
only if their covariance matrix is the identity matrix.
As usual, inference is based on an i.i.d. sample of size
$m$, denoted $X_1,\dots, X_m$ with
$X_t = (X_{t,1}, \dots, X_{t,n}) \in\R^n$ for $t=1,\ldots,m$. As
stated above, we assume that $\E X_{t,i} = 0$ and
$\Var(X_{t,i}) = 1$, and let $\sigma_{i,j} = \Cov(X_{t,i}, X_{t,j})$.

We are interested in testing whether the population covariance matrix
is the identity matrix, or not, so the null hypothesis is
\[
H_0 \dvt  \sigma_{i,j} = 0\qquad  \forall i \neq j .
\]
This testing problem is well studied in the classical regime where the
dimension $n$ is fixed and the sample size $m$ increases to
infinity. See, for example, Muirhead \cite{MR652932}, Section~8.4, where
the generalized likelihood ratio test (GLRT) -- against the alternative
hypothesis $H_1\dvt  \sigma_{i,j} \neq0$ for some $i \neq j$ -- is studied
in detail, as well as some unbiased variant.
When the dimension is large (i.e., $n \to\infty$), the GLRT may be
degenerate. This is discussed in detail in
Ledoit and Wolf \cite{MR1926169}, where other tests -- including a
new one -- are examined for consistency in this high-dimensional
regime. Their ideas are further explored in
Srivastava \cite{MR2328427}, Fisher \cite{Fisher11} and Chen, Zhang
and Zhong \cite{MR2724863}. All these tests are based on symmetric
polynomials of the sample correlation coefficients. We discuss this in
Section~\ref{sec:squared-sum}. These tests are shown to be consistent
when $n/m
\to c \in(0, \infty)$ under additional mild conditions.

While these papers focus on general consistency, our focus is on
alternatives where the covariance matrix is sparse, meaning that even
under the alternative hypothesis, only a few variables are
substantially correlated.
This sparse setting has been investigated in the last few years,
with recent work on the estimation of sparse covariance
matrices, see El Karoui \cite{MR2485011}, Bickel and
Levina \cite{MR2485008,MR2387969}
and Cai, Zhang and Zhou \cite{MR2676885}.
To our knowledge, testing for sparse correlation structures in a
multivariate sample has not been considered in detail before. This is
what we study here.

\subsection{Correlation models}
\label{sec:models}

We introduce sparse models of correlation matrices to test against.
Though many more models are possible, we choose a few emblematic examples
that are of interest in a much wider sense within the literature on
sparse covariance estimation. 
In all cases, the null hypothesis is that the observed vector
has identity covariance matrix. For the alternative hypothesis, we
consider the following prototypical examples:

\begin{itemize}
\item\textit{Block model}.
The covariance under the alternative hypothesis is the
identity matrix except for a $k \times k$ block on the diagonal.
Formally, given $\rho> 0$, we assume here that there is a subset of
indices of the form $S = \{i, \dots, i+k-1\}$ modulo $n$ -- for
aesthetic reasons -- such
that $\sigma_{i,j} \geq\rho$ if $i,j \in S, i \neq j$.
The set $S$ is called the \textsl{anomalous set}.
\item\textit{Clique model}. This model is defined as the block model
with the possible anomalous set $S$ ranging over all the
subsets of indices of size $k$.
\item\textit{Perfect matching model}. Suppose $n$ is a perfect square
with $n = k^2$. Here the components of the observed vector $X$
correspond to edges of the complete bipartite graph on $2k$
vertices. The alternative hypothesis is that the bipartite graph
has a perfect matching such that $\sigma_{i,j} \geq\rho$ for all
$i,j \in S, i \neq j$ where $S$ is the anomalous set of indices corresponding
to the edges of the perfect matching.
\end{itemize}
All through, we assume that $k < n/2$.
In fact, we will most interested in the regime where $k = \oo(n)$, where
the anomaly only affects a negligible fraction of the variables.

The block model is closely related to the models used
in Cai, Zhang and Zhou \cite{MR2676885} to obtain bounds on the
minimax risk
of estimating sparse matrices. Roughly speaking, Cai, Zhang and Zhou
use the block model with $S = \{1, \dots, k\}$ and place nonzero entries
in a (carefully designed) fashion within that block. The fraction of
nonzero entries within the block is about one-half. We could also
assume that only a fraction of the entries in the block are
nonzero and it would only change constants later on. More
importantly, to make the detection problem interesting, we need to
consider all possible blocks. Note that the block model is
parametric.
The clique model is a natural generalization of the block model
leading to a nonparametric model.
%
The perfect matching model gives an example of a class of sets
with a more intricate combinatorial structure which our approach
is able to deal with.

\subsection{Tests and their risks}
\label{sec:risk}

As usual, a \textsl{test} is a binary-valued function $f\dvtx \R^{nm} \to
\{0,1\}$, with $f(X_1, \dots, X_m)=1$ meaning that the test rejects
the null hypothesis $H_0$ in favor of the particular alternative
hypothesis of interest. We measure the performance of a test based on
its \textsl{worst-case risk} over the model of interest $\cM$,
formally defined by
\[
R^{\max}(f)= \PROB_0\bigl\{f(X_1, \dots,
X_m) =1\bigr\} + \sup_{M \in\cM} \PROB_M
\bigl\{f(X_1, \dots, X_m)= 0\bigr\} ,
\]
where $\PROB_0$ denotes the distribution under the null hypothesis,
while $\PROB_M$
denotes the distribution under the alternative hypothesis associated
with a
particular covariance structure $M$.
In our setup, $R^{\max}(f)$ depends on $n,m,\rho$, and the class
$\cC$ of possible index sets $S \subset\{1, \ldots, n\}$.
When all non-zero covariances $\sigma_{i,j}$
are actually equal to the lower bound $\rho$,
then $\PROB_M$ is determined by $S$ and,
with a slight abuse of notation, we write
$\PROB_S$ for $\PROB_M$. Clearly,
\[
R^{\max}(f) \ge\PROB_0\bigl\{f(X_1, \dots,
X_m) =1\bigr\} + \max_{S \in\cC} \PROB_S
\bigl\{f(X_1, \dots, X_m)= 0\bigr\}
\]
and indeed all lower bounds derived in this paper start with this inequality.
We will derive upper and lower bounds for the \textsl{minimax risk},
\[
R_*^{\max} := \inf_{f} R^{\max}(f) ,
\]
where the infimum is taken over all measurable functions
$f\dvtx \R^{nm} \to\{0,1\}$.

The lower bounds will be obtained by putting a prior on
model $\cC$ and obtaining a lower bound on the corresponding
\textsl{Bayesian risk} which never exceeds the worst-case risk. In all
cases, we draw the set $S$ uniformly at random within the
class $\cC$. The upper bounds are obtained by studying the
performance of specific tests.

We focus on the case where the dimension $n$ and the sample size $m$
are both large. Of course, such asymptotic statements only make sense
if we define sequences of integers $m=m_n$, $k=k_n$, positive
reals $\rho=\rho_n$, and
classes $\cC=\cC_n$. This dependency in $n$ will be left implicit.
In this asymptotic setting, we say that \textsl{reliable} detection is
possible (resp. impossible) if $R_*^{\max}\to0$ (resp. $\to1$) as
$n \to\infty$. Also we say that a sequence of tests $(f_n)$ is
asymptotically powerful (resp. powerless) if $R^{\max}(f_n) \to0$
(resp. $\to1$).

\subsection{A preview of results for the clique model}

Among the models we consider, the clique model is perhaps the most
compelling because of its relevance in applications and its
complexity.
Also, for a given value of $k$, the clique model is the richest possible
and therefore for any given $\rho,n,m$, $R_*^{\max}$ is larger than
for any other model. This makes the clique model an important benchmark.

Here we summarize our main findings for this special class.
We discover various types of behavior in distinct ranges
of the parameters $n,m,k,\rho$.
Roughly speaking, and ignoring logarithmic factors, we arrive at
the following conclusions. Two tests are competing for
near-optimality. The first one is a ``global'' test akin to the
classical test Muirhead \cite{MR652932}, Section~8.4, and the refinements
in Chen, Zhang and Zhong \cite{MR2724863} and Ledoit and Wolf \cite
{MR1926169}. The second is a ``local'' test reminiscent to
the generalized likelihood ratio test. The latter dominates the
former when
\[
\max \biggl(\frac{k^{3/2}}{n}, \frac{k^2}{n \sqrt{m}} \biggr) 
\]
is small, corresponding to smaller values of $k$.

Our results also uncover an interesting phase-transition phenomenon as
the sample size increases.
Indeed, we find that, when the sample size $m$ is sub-logarithmic in
the dimension $n$ -- meaning $m = \oo(\log n)$ -- reliable testing is only
possible if $\rho\to1$ or $k^2/n \to\infty$, which is what we found
in our previous work \cite{ArBuLu11} when $m = 1$.
The situation becomes drastically different when $m$ is at least a sufficiently
large constant multiple of $\log n$, where we learn that reliable
detection becomes possible in some settings where $\rho\to0$ and
$k=2$ -- as long as $\rho\to0$ sufficiently slowly.

%
%

\subsubsection{Computational considerations}

The ``local'' test that achieves near-optimal behavior in a large
range of the parameters is a scan statistic that requires the
computation of a maximum over all ${n \choose k}$ subsets of components
of size~$k$. In its naive implementation, this test is
computationally intractable, unless $k$ is very small.
We also believe that computing this test is a fundamentally hard
computational problem. We do not have a rigorous argument
to prove such a hardness result but it is worth pointing out
that the problem is quite similar, in spirit, to the notoriously
difficult \textsl{hidden clique problem}, see Alon, Krivelevich and
Sudakov \cite{AlKrSu99}.

What performance can we achieve with limited computational power?
Such questions of trade-off between statistical performance and
computational complexity are at the heart of high-dimensional
statistics and machine learning. We probe this question and describe
a family of tests that balances detection performance and
computational complexity.

In particular, in Section~\ref{sec:maxsum} we design a test
that achieves near-optimal performance (similar to that of the scan
statistic)
and may be computed in polynomial-time in $n$ when $m = \OO(\log n)$,
$\rho$ is a constant, and $k\sim n^a$ for some $a\in(0,1/2)$.
In Section~\ref{sec:SDP}, we discuss another computationally efficient test
based on a convex relaxation of the local test.

\subsubsection{An application in the study of random geometric graphs}

In Section~\ref{sec:cliquergg}, we apply the lower bound for the
optimal risk in the clique model in a perhaps unexpected context
and derive a new lower bound for the clique number of a
high-dimensional random geometric graph. The setup is as follows.

Consider a random geometric graph on the unit sphere in dimension $m$.
The graph has $n$ vertices, each corresponding to a random point on
the unit sphere. Two vertices are connected by an edge if the inner
product of the corresponding points is positive. In a recent paper,
Devroye, Gy\"orgy, Lugosi and Udina \cite{DeGyLuUd11} studied the
clique number (i.e., the size of the
largest clique in the graph) $\omega(n,m)$ of such a graph in various
regimes. They showed that when $m \sim c\log n$ for a sufficiently
small constant $c$, $\omega(n,m)=n^{1-\oo(1)}$ with high probability,
while when $m \ge9 \log^2 n$, $\omega(n,m) = \OO(\log^3 n)$. However,
nothing was known about the behavior of the clique number in
between. In particular, it was unclear where exactly the clique number
becomes polylogarithmic. In Section~\ref{sec:cliquergg}, we show that the
phase transition occurs at $m \asymp\log^2 n$.
(We use the notation $a_n \asymp b_n$ when $(a_n)$ and $(b_n)$ are two
sequences such that $a_n = \OO(b_n)$ and $b_n = \OO(a_n)$.)
In particular, we
prove that for all $c>0$, when $m \sim c\log n$, then the median of
$\omega(n,m)$ grows as a positive power of $n$ and even for $m \asymp
\log^{2-\epsilon} n$, the median of $\omega(n,m)$ grows faster than
any power of $\log n$, for all $\epsilon>0$.

\subsection{More related work}

As mentioned before, the literature on \textsl{sparse} covariance
estimation
has become quite extensive. In spite of this surge of interest in sparse
high-dimensional models, not much has been done in terms of detection
of correlations.
We note the work
of Verzelen and Villers \cite{MR2604699}, who consider the task of
testing a given
dependency structure. Our objective here is admittedly more modest
and a more closely related is our own paper \cite{ArBuLu11}, which
focuses entirely on the case where the sample size is equal to one
(i.e., $m = 1$). Our results here are seen to extend those in the
one-sample case, with the regimes now partitioned according to the
sample size.

Note that our work is different from Butucea and Ingster \cite
{ingster} where the task is
the detection of a submatrix with higher per-coordinate mean in a
large matrix with i.i.d. Gaussian entries, which is more closely
related to the literature on the detection of sparse nonzero entries
in the mean of a random vector. Our work has parallels with that
literature which, for the clique model, focuses on the
``detection-of-means''
problem
(see Jin \cite{JinPhD}, Ingster \cite{Ingster99},
Baraud \cite{baraud},
Donoho and Jin \cite{dj04}, Hall and Jin \cite{hj09},
Arias-Castro, Cand{\`e}s, Helgason and Zeitouni \cite{maze},
Addario-Berry, Broutin, Devroye and Lugosi \cite{combin})
defined as follows:
Under the null hypothesis,
the vectors $X_t$ are i.i.d. standard normal, while under the
alternative hypothesis, there is a subset $S \subset\{1, \ldots, n\}$
in some class $\cC$ of interest
such that the $X_t$ are i.i.d. normal with mean $(\mu_1, \dots, \mu_n)^T$
and identity covariance, where $\mu_i \ge\mu$ for $i \in S$ and
$\mu_i=0$ for $i \notin S$. Thus,
$\mu> 0$ is the minimum (per-coordinate) signal amplitude. Of
course, one immediately reduces by sufficiency to the case $m=1$ by
averaging over the sample. This explains why the literature focuses on
the case $m=1$.
The connection between the detection-of-means problem with
the correlation detection problem studied here
was detailed (for $m=1$) in our previous paper \cite{ArBuLu11}, where
$\rho$ was found to correspond to $\mu^2$.
The connection is based on the following simple
representation of equi-correlated normal random variables.

\begin{lem}[(Berman \cite{MR0145564})]
\label{lem:represent}
Let $X_1, \dots, X_k$ be standard normal random variables with
$\Cov(X_i,\allowbreak  X_j) = \rho>0$ for $i \neq j$. Then there are
independent standard normal random variables $V, Y_1, \dots, Y_k$
such that $X_i = \sqrt{\rho} V + \sqrt{1-\rho} Y_i$ for all $i$.
\end{lem}

Thus, given $V$, the problem becomes that of detecting a subset of
variables -- here implicitly assumed to be indexed by
$S = \{1,\ldots, k\}$ -- with nonzero mean (equal to $\sqrt{\rho} V$) and
with a variance equal to $1-\rho$ (instead of $1$). This representation
was used in \cite{ArBuLu11} to obtain a general lower bound that
seemed otherwise out of reach of more standard methods based on the
second moment of the likelihood ratio.

This connection with the detection-of-means problem
also applies in the case where $m > 1$, but with a
twist. Indeed, when detecting correlations one does not average the
vectors $X_t$ but their covariances. So a simple reduction to the
case $m=1$ does not apply. However,
one may still apply the representation result Lemma~\ref{lem:represent}
to each
observation vector $X_t$, yielding $V_t$'s and $Y_{t,i}$'s that are
independent standard normal random variables.
By conditioning on $V_1, \dots, V_m$, the problem becomes equivalent
to detecting a subset of variables with means $\sqrt{\rho} V_t,
t=1,\dots,m$. What makes the situation more complex is that
the signs of the $V_t$'s are
random. Our approach to finding a general lower bound is based
on this representation without which more standard
methods seem to fail. The general lower bound, which is the key
technical
result of this paper, is given in Theorem~\ref{thm:lower} below.

\subsection{Contribution and content of the paper}

We obtain a general lower bound in Section~\ref{sec:lower} akin to,
but not a
straightforward extension of, the lower bound we obtained in
\cite{ArBuLu11}. We then study a number of tests that are
near optimal in the sense that they come close to achieving
the detection lower bound for various models.
This is done in Section~\ref{sec:tests}.
We then specialize these general results in Sections~\ref{sec:clique},
\ref{sec:block} and \ref{sec:matching}, to the three models described in
Section~\ref{sec:models}.
We also discuss computational issues, particularly in the clique
model.
In Section~\ref{sec:cliquergg}, we apply our general lower bound
to the problem of studying the size of the clique number of a random
geometric graph on a high-dimensional sphere. We close the paper with
a discussion in Section~\ref{sec:discussion} of possible extensions and
challenges.


\section{Lower bounds}
\label{sec:lower}

In this section, we derive a general lower bound for the minimax risk
$R_*^{\max}$. As mentioned in Section~\ref{sec:risk}, the first step
is to restrict the supremum in the definition of $R^{\max}(f)$
to covariance matrices in which all the nonzero
entries are equal to $\rho> 0$ and then lower bound the maximum
by an average. In particular, we have $R_*^{\max} \ge R^*$ where
$R^* = \inf_f R(f)$ and
\[
R(f) \defeq\PROB_0\bigl\{f(X_1, \dots,
X_m) =1\bigr\} + \frac{1}{|\cC|} \sum_{S \in\cC}
\PROB_S\bigl\{f(X_1, \dots, X_m)= 0\bigr\} .
\]
Note that $R^*$ is just the Bayes risk for the uniform prior
on the models $S \in\cC$.
It is well known that the test $f^*$ that achieves the infimum
(i.e., $R(f^*)=R^*$) is the \textsl{likelihood ratio test} with
critical value 1, and there is a whole machinery that can be used to
bound that risk from below.

The following lower bound has a similar flavor as the main result in
our previous work \cite{ArBuLu11}. In particular, we make appear some
moment of $Z$, a random variable that represents the size of the
overlap of two index sets taken at random from the class $\cC$. A
straightforward adaptation of the arguments we used in \cite
{ArBuLu11} leads to a lower bound in terms of the moment generating
function of $Z$. Here, unfortunately, this quantity is too large to
obtain sharp results in most regimes. The key contribution of the
following result is to replace the exponential function by the
hyperbolic cosine. This allows us to derive much sharper results,
essentially because around $0$ one has $\exp(x) - 1 \sim x$ while
$\cosh(x) - 1 \sim\frac{x^2}{2}$.
%
%
\begin{thm}
\label{thm:lower}
For any class $\cC$, any $\rho\in(0, 1)$, and any $a \ge\sqrt{3}$,
%
\begin{equation}
\label{lower} R^* \geq\P \bigl(\chi_m^2 \le m
a^2 \bigr) \biggl(1 - \frac{1}2 \sqrt{\E\min \bigl[\exp (m
\nu_a Z ), \cosh^m(\xi_a Z) \bigr] - 1}
\biggr),
\end{equation}
where
\[
\nu_a := \frac{\rho a^2}{1+\rho} - \frac{1}2 \log\bigl(1-
\rho^2\bigr)\quad  \mbox{and}\quad  \xi_a := \frac{\rho a^2}{1-\rho^2},
\]
and where $\chi_m^2$ has chi-squared distribution with $m$ degrees of
freedom, and $Z=|S \cap S'|$ with $S, S'$ i.i.d. uniform from $\cC$.
\end{thm}
%
%
\begin{pf}
Following \cite{ArBuLu11} we start the proof by using the
representation of the data given by Lemma~\ref{lem:represent}.
Under the alternative hypothesis $H_1$, $X \in\R^{m \times n}$ can be
written as
%
\begin{equation}
\label{eq:representation} X_{t,i} = \lleft\{ %
\begin{array} {l@{\qquad}l}
Y_{t,i} & \mbox{ if } i \notin S, t \in[m],
\\
\sqrt{\rho} V_t + \sqrt{1-\rho} Y_{t,i} & \mbox{ if } i
\in S, t \in[m],
\end{array} %
\rright.
\end{equation}
where, for any positive integer $m$, $[m] := \{1, \dots, m\}$, and
$(Y_{t,i})_{i \in[n], t \in[m]}, (V_t)_{t \in[m]}$ are i.i.d.
standard normal random variables.

In our previous paper \cite{ArBuLu11}, we conditioned on $V := (V_1,
\dots, V_m)$.
Here, instead, we condition on $U := (U_1, \dots, U_m)$, where $U_t := |V_t|$.
Let $\eps_t = \operatorname{sign}(V_t)$, which are i.i.d. Rademacher.
We consider now the alternative hypothesis $H_1(u)$, defined as the
alternative hypothesis
$H_1$ given $U = u \in\bbR^m$. Let $R(f)$, $L$, $f^*$
(resp. $R_u(f)$, $L_u$, $f_u^*$) be the risk of a test $f$, the
likelihood ratio, and the optimal (likelihood ratio) test, for $H_0$
versus $H_1$ (resp. $H_0$ versus $H_1(u)$). For any $u \in\R^m$,
$R_u(f_u^*) \leq R_u(f^*)$, by the optimality of $f_u^*$ for $H_0$
vs. $H_1(u)$. Therefore, conditioning on $U$,
\[
R^* = R\bigl(f^*\bigr)= \E_{U} R_U\bigl(f^*\bigr) \geq
\E_{U} R_U\bigl(f_U^*\bigr) = 1 -
\tfrac
{1}{2} \E_{U} \E_0 \bigl|L_U(X) - 1\bigr| .
\]
($\E_U$ is the expectation with respect to $U$.)
Using the fact that
$\E_0 |L_u(X) - 1| \le2$ for all $u$, we have (with $B(0,a)$ being
the euclidean ball of radius $a$ in $\R^m$)
\[
\E_{U} \E_0 \bigl|L_U(X) - 1\bigr| \le2\PROB\bigl\{\Vert U
\Vert >a \sqrt{m}\bigr\} + \PROB\bigl\{\Vert U\Vert \le a \sqrt{m}\bigr\} \max
_{u \in B(0,a \sqrt{m})} \E_0 \bigl|L_u(X) - 1\bigr| .
\]
Therefore, using the Cauchy--Schwarz inequality,
\begin{eqnarray*}
1 - \frac{1}{2} \E_{U} \E_0 \bigl|L_U(X) -
1\bigr| & \ge& \PROB\bigl\{\Vert U\Vert \le a \sqrt{m}\bigr\} \biggl( 1-\frac{1}{2}\max
_{u \in B(0,a
\sqrt{m})} \E_0 \bigl|L_u(X) - 1\bigr| \biggr)
\\
& \geq& \PROB\bigl\{\Vert U\Vert \le a \sqrt{m}\bigr\} \biggl(1 -
\frac{1}2 \max_{u \in
B(0,a \sqrt{m})} \sqrt{\E_0
L_u^2(X) - 1} \biggr) .
\end{eqnarray*}

We turn our attention to bounding $\E_0 L_u^2(X)$ from above. Let
$L_{u, \eps, S}(x)$ denote the likelihood ratio when $S$ is anomalous,
given $u$ and $\eps$, which is equal to
\[
L_{u, \eps, S}(x) = \frac{1}{(1-\rho)^{m k/2}} \exp \Biggl(\sum
_{t=1}^m \sum_{i \in
S}
\frac{x_{t,i}^2}{2} - \frac{(x_{t,i} - \sqrt{\rho} \eps_t
u_t)^2}{2 (1-\rho)} \Biggr).
\]

Since $L_u(x) = \E_\eps\E_S L_{u, \eps, S}(x)$, by Fubini's
theorem, we have
\[
\E_0 L_u(X)^2 = \E_{S, S'}
\E_{\eps,\eps'} \E_0 L_{u, \eps,
S}(X) L_{u, \eps', S'}(X),
\]
where $\eps, \eps'$ are i.i.d. Rademacher vectors and $S,S'$ are
i.i.d. uniform in the class $\cC$. We have
\[
L_{u, \eps, S}(x) L_{u, \eps', S'}(x) = (1-\rho)^{-m k} \exp
\bigl(H_1(x) + H_2(x) + H_3(x)\bigr),
\]
where
\begin{eqnarray*}
H_1(x) &:=& \sum_{t=1}^m \sum
_{i \in S \cap S'} x_{t,i}^2 -
\frac
{(x_{t,i} - \sqrt{\rho} \eps_t u_t)^2}{2(1-\rho)} - \frac{(x_{t,i}
- \sqrt{\rho} \eps'_t u_t)^2}{2(1-\rho)},
\\
H_2(x) &:=& \sum_{t=1}^m \sum
_{i \in S \setminus S'} \frac
{x_{t,i}^2}{2} - \frac{(x_{t,i} - \sqrt{\rho} \eps_t
u_t)^2}{2(1-\rho)},
\\
H_3(x) &:=& \sum_{t=1}^m \sum
_{i \in S' \setminus S} \frac
{x_{t,i}^2}{2} - \frac{(x_{t,i} - \sqrt{\rho} \eps'_t
u_t)^2}{2(1-\rho)} .
\end{eqnarray*}
Let $Z = |S \cap S'|$.
We see that $H_1(X), H_2(X), H_3(X)$ are independent of each other
under the null hypothesis with
\[
\E_0 \exp\bigl(H_2(X)\bigr) = (1-\rho)^{m |S \setminus S'|/2} =
(1-\rho)^{m
(k - Z)/2}
\]
and similarly for $\E_0 \exp(H_3(X))$, while
\[
\E_0 \exp\bigl(H_1(X)\bigr) = \biggl(
\frac{1-\rho}{1+\rho} \biggr)^{m Z / 2} \exp \Biggl(\frac{\rho Z}{1-\rho^2} \sum
_{t=1}^m \eps_t
\eps'_t u_t^2 -
\frac{\rho^2 Z}{1-\rho^2} \| u\|^2 \Biggr).
\]
For the latter, we used the fact that $\eps_t^2 = {\eps_t'}^2 = 1$,
to get
\begin{eqnarray*}
&& \int_\R\exp \biggl(x_{t,i}^2 -
\frac{(x_{t,i} - \sqrt{\rho} \eps_t u_t)^2}{2(1-\rho)} - \frac{(x_{t,i} - \sqrt{\rho} \eps'_t u_t)^2}{2(1-\rho)} \biggr) \exp\bigl(-x_{t,i}^2/2
\bigr) \frac{\mathrm{d}x_{t,i}}{\sqrt{2 \pi}}
\\
&&\quad = \int_\R\exp \biggl(\frac{\rho(\eps_t \eps'_t -\rho)
u_t^2}{1-\rho^2} -
\frac{1+\rho}{2(1-\rho)} \biggl(x_{t,i} - \frac{\sqrt{\rho}}{1+\rho} \bigl(
\eps_t + \eps'_t\bigr) u_t
\biggr)^2 \biggr) \frac{\mathrm{d}x_{t,i}}{\sqrt{2 \pi}}
\\
&&\quad = \sqrt{\frac{1-\rho}{1+\rho}} \exp \biggl(\frac{\rho(\eps_t \eps'_t -\rho) u_t^2}{1-\rho
^2} \biggr) ,
\end{eqnarray*}
where the last line comes from a simple change of variables.
Hence,
\[
\E_0 L_{u, \eps, S}(X) L_{u, \eps', S'}(X) = \bigl(1-
\rho^2\bigr)^{-mZ / 2} \exp \Biggl(\frac{\rho Z}{1-\rho^2} \sum
_{t=1}^m \eps_t
\eps'_t u_t^2 -
\frac{\rho^2 Z}{1-\rho^2} \|u\|^2 \Biggr). 
\]
Let $\xi= \xi_1 = \rho/(1-\rho^2)$.
Since $(\eps_t \eps'_t \dvt  t = 1, \dots, m)$ are i.i.d. Rademacher, we have
\[
\E_{\eps, \eps'} \exp \Biggl(\xi Z \sum_{t=1}^m
\eps_t \eps'_t u_t^2
\Biggr) = \prod_{t=1}^m \cosh \bigl(\xi Z
u_t^2 \bigr) ,
\]
so that 
\[
\E_{\eps, \eps'} \E_0 L_{u, \eps, S}(X) L_{u, \eps', S'}(X) =
\bigl(1-\rho^2\bigr)^{-mZ / 2} \prod
_{t=1}^m \cosh\bigl(\xi Z u_t^2
\bigr) \exp\bigl(- \rho \xi Z u_t^2\bigr) .
\]
%
Holding $Z \ge1$ fixed, we maximize this over $\|u\|^2 = \sum_t u_t^2
\le a^2 m$ using Lagrangian multipliers and checking the
Karush--Kuhn--Tucker conditions, finding that at a local maximum all
$u_t^2$ must be equal.
Hence,
%
\begin{eqnarray}
\label{max0}\prod_{t=1}^m \cosh\bigl(\xi Z
u_t^2\bigr) \exp\bigl(- \rho\xi Z u_t^2
\bigr) &\le& \Bigl(\max_{0 \leq c \leq\xi_a Z} \cosh(c) \exp(- \rho c)
\Bigr)^m
\\
\label{max1}&=& \max \bigl(1, \cosh^m(\xi_a Z) \exp(- m \rho
\xi_a Z) \bigr),
\end{eqnarray}
where the last equality comes from the fact that the function
$h_\rho(c) := \cosh(c) \exp(- \rho c)$ is decreasing on $(0, \rho)$ and
increasing on $(\rho, \infty)$, so that its maximum over $[0,\xi_a
Z]$ is either at
$c = 0$ or $c = \xi_a Z$.
Straightforward calculations lead to
\[
h_\rho(c) > 1 \quad \Leftrightarrow\quad  g(c) := \frac{1}c \log\cosh(c)
> \rho.
\]
Since $g(c) > 1 - \frac{1}c \log2$, the maximum of $h_\rho(c)$ over $c
\in[0,\xi_a Z]$ is at $c = \xi_a Z$ when
\[
1 - \frac{1}{\xi_a Z} \log2 \ge\rho\quad \Leftrightarrow\quad  a^2 Z \ge
\frac
{1+\rho}\rho\log2.
\]
Since we consider $Z \ge1$, the last inequality is true if $\rho\ge
1/2$ and $a^2 \ge3 \log(2)$.
This inequality is far off when $\rho$ is small, so we need to derive
another bound.
Noting that $g$ is seen to be strictly increasing on $(0,\infty)$ with
range $(0,1)$, $w(\rho) := g^{-1}(\rho)$ is well defined and, as a
function of $\rho$, is infinitely differentiable and strictly increasing.
Elementary calculations show that
\[
w' = \frac{w}{\tanh(w) - \rho},\qquad  w'' =
\frac{w \tanh^2(w)}{(\tanh
(w) -\rho)^2},
\]
implying in particular that $w$ is convex.
Also, since $g(c) \ge\frac{1}c \log(1 + \frac{c^2}2)$ -- because $\mathrm{e}^c
\ge1 + c + c^2/2$ for all $c \ge0$ and $\mathrm{e}^c \ge1 + c$ for all $c \in
\bbR$ -- we have $g(2) \ge\frac{1}2 \log(3) > \frac{1}2$ and therefore
$w(1/2) < 2$.
Hence, by convexity of $w$, we have $w(\rho) \le w(1/2) \rho< 2 \rho
$ for all $\rho\le1/2$.
Now, the maximum in \eqref{max0} at $c = \xi_a Z$ when
\[
w(\rho) < \xi_a Z \quad \Leftarrow\quad 2 \rho\le\frac{\rho a^2}{1-\rho^2} Z
\quad \Leftrightarrow\quad 2 \bigl(1-\rho^2\bigr) \le a^2 Z.
\]
When $Z\ge1$, the latter is true when $\rho\le1/2$ and $a^2 \ge3/2$.
Hence, given that $a^2 \ge3$ by assumption, the maximum in \eqref
{max0} at $c = \xi_a Z$.

We therefore arrive at
\[
\max_{u \in B(0,a \sqrt{m})} \E_0 L_u^2(X)
\le\E \bigl( \bigl(1-\rho ^2\bigr)^{-mZ / 2}
\cosh^m(\xi_a Z) \exp(- m \rho\xi_a Z) \bigr)
.
\]
We then have
\[
\cosh^m(\xi_a Z) \exp(- m \rho\xi_a Z) \le
\exp\bigl(m (1-\rho) \xi _a Z\bigr) = \exp \biggl( \frac{m \rho a^2}{1+\rho}
Z \biggr).
\]
And also
\begin{eqnarray*}
\bigl(1-\rho^2\bigr)^{-mZ / 2} \exp(- m \rho\xi_a
Z) &\le&\exp \biggl(- \frac{m Z}{2 (1 -\rho^2)} \bigl[2 \rho^2 + \bigl(1-
\rho^2\bigr) \log\bigl(1-\rho^2\bigr) \bigr] \biggr)
\\
&\le&1,
\end{eqnarray*}
where in the first line we used $a^2 \ge1$ and in the second line the
fact that $s + \frac{1-s}{2} \log(1-s) \ge0$ for all $s \in(0,1)$.
With this, we conclude.
\end{pf}

In Sections~\ref{sec:clique}, \ref{sec:block}, and \ref{sec:matching},
we specialize Theorem~\ref{thm:lower} to the
different models we described in Section~\ref{sec:models}.

\section{Tests}
\label{sec:tests}

In this section, we introduce and briefly discuss two natural tests
that will be seen to perform near optimally in various regimes of the
parameters. This optimality property will be established in Sections~\ref{sec:clique}, \ref{sec:block} and \ref{sec:matching}, by
comparing simple performance bounds with the implications of
Theorem~\ref{thm:lower}.

The first test, that we call ``squared-sum test'',
is based on a global test statistic that
does not take the class $\cC$ into account at all.

The second test, a ``localized'' squared-sum test,
is based on a simple scan statistic. It may also be interpreted
as a simplified version of the generalized
likelihood ratio test.

As we will see, one of the two tests above always has a near-optimal
performance in all three specific classes we discuss.
Thus, the story is
essentially complete for the point of view of detection performance.
Unfortunately, when the class $\cC$ is large -- as in the clique
model --, the localized squared-sum test is computationally unfeasible,
at least in its naive implementation. We discuss two possible
substitutes. The first one is
a simple ``maximum correlation test'' that turns out to be nearly
optimal for very small values of $k$. In Section~\ref{sec:maxsum}, we discuss
another test in the context of the clique model that is both
near-optimal and computationally feasible when the sample size is at
most logarithmic in the dimension $n$.
In Section~\ref{sec:SDP}, a conceptually different computationally efficient
alternative is discussed.

All performance bounds derived below are in terms of the average correlation
%
\begin{equation}
\label{Rho} \Rho= \frac{1}{k(k-1)} \sum_{i,j \in S: i \neq j}
\sigma_{i,j} \ge \rho,
\end{equation}
where $S$ is the anomalous set.

\subsection{The squared-sum test}
\label{sec:squared-sum}

Let $\Sigma$ denote the covariance matrix of the distribution of
$X_1$. When
the alternative hypothesis is simply $H_1\dvt  \Sigma\neq I$ (where $I$ is the
identity matrix), without any sign restriction on the entries of
$\Sigma$, one of the simplest tests is that of Nagao \cite{MR0339405},
which is based on the Frobenius norm of the difference between the
sample covariance matrix $\widehat\Sigma$ and the identity matrix
$I$.
Nagao's test is based on the test statistic
\[
\frac{1}n \|\widehat\Sigma- I\|_F^2 =
\frac{1}n \trace\bigl[(\widehat \Sigma- I)^2\bigr].
\]
We also refer to Schott \cite{MR2234197}, who (like us) assumes that the
variables have unit variance under the alternative hypothesis.
Ledoit and Wolf \cite{MR1926169} show that this test is not always
consistent against fixed alternatives.
They, and others including Srivastava \cite{MR2328427}, Fisher
\cite{Fisher11} and Chen, Zhang
and Zhong \cite{MR2724863}, suggest variants based on consistent
estimates for the Frobenius norm $\trace[(\Sigma- I)^2]$.

Given that we know that the variances are equal to 1 and the
correlations are non-negative under the alternative hypothesis, it is
more natural to consider the test that rejects for large\vspace*{1pt} values of
$\sum_{i < j} \widehat{\sigma}_{i,j}$, where $\widehat\Sigma=
(\widehat{\sigma}_{i,j})$. For simplicity, we consider instead the
\textsl{squared-sum test} that rejects for large values of the test statistic
%
\begin{equation}
\label{squared-sum} Y = \sum_{t=1}^m \Biggl(
\sum_{i=1}^n X_{t,i}
\Biggr)^2 = m \sum_{i,j
= 1}^n
\widehat{\sigma}_{i,j} .
\end{equation}
The two tests are thus closely related.
In fact, one may easily check that they have similar asymptotic power
properties. Our preference for the second test is only for convenience.

The following result gives a simple characterization of the
performance of the squared-sum test. Since the test does not use
information about the class $\cC$, its minimax risk does not depend
on the model either.

\begin{prp} \label{prp:squared-sum}
The squared-sum test that rejects $H_0$ when
$Y \ge n (m + a \sqrt{m})$ is asymptotically powerful when
$a \to\infty$ such that
$a \le\frac{1}2 \sqrt{m}$ and $\Rho\sqrt{m} k^2/n \ge5 a$.
The squared-sum test with any threshold value for $Y$ is asymptotically
powerless
when $\Rho\sqrt{m} k^2/n \to0$.
\end{prp}

To interpret this result, note that in a very high-dimensional setting
(i.e., when $m$ of smaller order than any positive power of $n$), and
with constant correlation (i.e., $\Rho$ is a numerical constant), the
proposition states that the squared-sum test is powerful in a
non-sparse regime where $k \gg\sqrt{n}$. On the other hand in sparse
situations with constant correlation, the sample size must be very
large for the squared-sum test to be powerful, as one needs $m \gg
n^2/k^4$. We will see in the next section that in this sparse regime
other tests perform much better.
\begin{pf*}{Proof of Proposition \ref{prp:squared-sum}}
Under the null hypothesis, $Y \sim n \chi^2_m$, and therefore
\[
\P_0\bigl(Y > n (m + a \sqrt{m})\bigr) \to0.
\]
Under the alternative hypothesis, $Y \sim(1+b) n \chi^2_m$, where $b
:= \Rho
k(k-1)/n$, and hence
\[
\P_1\bigl(Y \le(1+b) n (m - a\sqrt{m})\bigr) \to0,\qquad  n \to\infty.
\]
Using the assumptions on $a$ and $b$, we have
\begin{eqnarray*}
(1+b) n (m - a \sqrt{m}) - n (m + a \sqrt{m}) &\ge& n \sqrt{m} \bigl(b \sqrt{m}
- (2+b) a \bigr)
\\
&\ge& n \sqrt{m} \biggl(\frac{1}2 b\sqrt{m} - 2 a \biggr)
\\
&\ge& n \sqrt{m} \frac{a}2 > 0,
\end{eqnarray*}
and therefore
the test with critical value $n (m + a \sqrt{m})$ is asymptotically powerful.

Suppose that $b \sqrt{m} \to0$. We still have that $Y/n \sim
\chi^2_m$ under $H_0$ while $Y/n \sim(1+b) \chi^2_m$ under $H_1$. If
$m$ is fixed, $Y/n$ is asymptotically $\chi^2_m$ under the alternative
hypothesis
since $b \to0$ in this case. If $m \to\infty$, $(Y/n -
m)/\sqrt{2m}$ is asymptotically standard normal under both the null
and the alternative hypotheses, since under $H_1$,
\[
\frac{Y/n - m}{\sqrt{2m}} = \frac{Z - m}{\sqrt{2 m}} + \frac{b
Z}{\sqrt{2 m}} =
\frac{Z - m}{\sqrt{2 m}} + \OO_P(b \sqrt{m}),
\]
where $Z = Y/(n (1+b)) \sim\chi^2_m = \OO_P(m)$. We may apply Slutsky's
theorem to conclude.
\end{pf*}

\subsection{A localized squared-sum test}
\label{sec:local-squared-sum}

When $k$ is smaller, global tests such as the squared-sum test are not
very powerful. The generalized likelihood ratio
test ``scans'' over all subsets $S$ in the class $\cC$.
Instead of studying the generalized likelihood ratio test,
we consider a localized version of the
squared-sum test that has similar power and is a little easier to
analyze. The \textsl{localized squared-sum test} rejects the
null hypothesis for large
values of
the test statistic
\[
Y_\mathrm{scan} = \max_{S \in\cC} Y_S, \qquad \mbox{where } Y_S = \sum_{t=1}^m
\biggl(\sum_{i \in S} X_{t,i}
\biggr)^2 = m \sum_{i,j \in S} \widehat{
\sigma}_{i,j} .
\]
The following result gives sufficient conditions for the test to
be asymptotically powerful. The conditions are in terms of the
cardinality of the class $\cC$. Sharper bounds that take into account
the fine metric structure of $\cC$ are also possible by more careful
bounding of the distribution of $Y_\mathrm{scan}$ under the
null hypothesis. However, as we will see below, this bound is already
quite sharp for the specific classes considered in this paper,
and to preserve relative simplicity of the arguments we do not
pursue sharper bounds here.

\begin{prp} \label{prp:local-squared-sum}
The localized squared-sum test that rejects the null hypothesis
if
$
Y_\mathrm{scan} > b k m$,
is asymptotically powerful when $\Rho k \ge3 (b-1)$,
$m\to\infty$,
and either
\[
\frac{\log|\cC|}m \to0 \quad \mbox{and}\quad  b \ge1 + \sqrt{\frac{5\log
|\cC|}m} ,
\]
or
\[
\frac{\log|\cC|}m \to\infty\quad \mbox{and}\quad  b \ge\frac{3 \log|\cC
|}m .
\]
\end{prp}

Note that the proposition makes a distinction between a \emph
{ultra-high} dimensional setting where $m \ll\log|\cC|$, and a
(potentially) high dimensional setting where $m \gg\log|\cC|$. In
the former case, the result states that $k$ needs to be large enough
for the localized squared-sum test to be powerful. On the other hand in
the other less extreme regime the correlation $\Rho$ can be as small
as $\frac{1}{k} \sqrt{\frac{\log|\cC|}m}$.

In particular, in the sparse regime where $k$ is a constant the result
above ensures that asymptotic power is achieved for $\Rho$ as small as
$\sqrt{\frac{\log n}{m}}$. In the non-sparse regime, for large
classes $\cC$ (i.e., such that $\log|\cC|$ is roughly $k \log n$),
asymptotic power is achieved in the ultra-high dimensional setting
(that is $m \ll k \log n$) for $\Rho$ as small as $\frac{\log n}{m}$,
while in the high dimensional setting (that is $m \gg k \log n$) it is
achieved for $\Rho$ as small as $\sqrt{\frac{\log n}{k m}}$.

Other consequences of this proposition will be discussed in the next sections.
\begin{pf*}{Proof of Proposition \ref{prp:local-squared-sum}}
Observe that under the null hypothesis $Y_S \sim k \chi^2_m$ for all
$S \in\cC$. By a simple Chernoff bound for the chi-square distribution,
for all $b > 1$,
%
\begin{equation}
\label{chi} \P \bigl(\chi_m^2 > b m \bigr) \le\exp
\biggl( - \frac{m}2 H(b) \biggr),
\end{equation}
where $H(b) := b - 1 - \log b$ for $b > 1$. Hence, by the union bound,
%
\begin{equation}
\label{lss1} \P_0(Y_\mathrm{scan} > b k m) \le|\cC| \exp
\biggl( - \frac{m}2 H(b) \biggr).
\end{equation}
%
When $(\log|\cC|)/m \to0$, using the fact that $H(b) \sim(b-1)^2/2$
when $b \to1$,
we see that the right-hand side in \eqref{lss1} tends to zero when $b
\ge1 + \sqrt{5 \log|\cC|/m}$.
When $(\log|\cC|)/m \to\infty$, using the fact that $H(b) \sim b$
when $b \to\infty$, so the right-hand side in \eqref{lss1} tends to
zero when $b \ge3 \log|\cC|/m$.

Under the alternative hypothesis where $S$ is anomalous -- that is,
under $\PROB_S$ --, $Y_S \sim(k + \Rho k (k-1)) \chi^2_m$, and therefore
%
\begin{equation}
\label{Yscan-H1} Y_\mathrm{scan} \ge Y_S > km \bigl(1 +
\Rho(k-1)\bigr) \bigl(1 - \OO_P(1/\sqrt{m})\bigr).
\end{equation}

Hence, the test is asymptotically powerful when
$(1 + \Rho(k-1)) (1 - \OO_P(1/\sqrt{m})) > b$,
and we can check that $\Rho k \ge3 (b-1)$ implies this in both regimes.
\end{pf*}

The case when
$\log|\cC|/m \asymp1$
can be dealt with in the same way, yielding that, with a proper choice
of threshold, the localized squared-sum test is asymptotically powerful when
%
\begin{equation}
\label{rho-local} \Rho k \ge A \max \biggl(\sqrt{\frac{\log|\cC|}m},
\frac{\log|\cC|}m \biggr) ,
\end{equation}
for a sufficiently large constant $A$.

When the class $\cC$ is large (i.e., has size exponential in $k$), the
test statistic $Y_\mathrm{scan}$ may be difficult to compute as it
involves solving
a nontrivial combinatorial optimization problem. This is the case
for the clique model (unless $k$ is very small) and the matching
model.

\subsection{Maximum correlation test}
\label{sec:pairwise}

Finally, we mention the possibly simplest test that one would think of
when confronted with testing $H_0$ in the sparse regime.
This is the test that rejects for large
values of the maximum pairwise empirical correlation
\[
Y_{\max}= \max_{i \neq j} \sum
_{t=1}^m X_{t,i} X_{t,j} .
\]
In fact, this test does have some power in the sparse regime, and is
actually near-optimal when $k$ is fixed as the following result shows.
However, one cannot expect a good performance of this test for large
values of $k$. An advantage of this test is that it may be computed
efficiently in
a straightforward manner.

\begin{prp} \label{cor:pairwise}
The maximum correlation test that rejects $H_0$ when
$Y_{\max} > \sqrt{5 m \log n}$
is asymptotically powerful when
%
\begin{equation}
\label{pairwise} \Rho\ge\sqrt{5 (\log n)/m} .
\end{equation}
%
\end{prp}

Note that the performance described in this proposition matches the one
of the localized squared-sum test for the sparse regime where $k$ is a
constant. Moreover, the performance is better than the squared-sum test
for $k$ is small enough, precisely when $k^2 \ll\sqrt{\log n} / n$.
%
\begin{pf*}{Proof of Proposition \ref{cor:pairwise}}
Assume that $m \ge5 \log n$ for otherwise the condition for $\Rho$ is vacuous.
For $i \neq j$ fixed, under the null hypothesis, $X_{t,i} X_{t,j}, t=
1, \dots, m$
are i.i.d. with zero mean, unit variance, and finite
moment generating function in a neighborhood of the origin.
In fact, it is equal to $(1 - \lambda^2)^{-1/2}$ for $\lambda\in(-1,1)$.
Hence, by a standard result on moderate deviations, see, for example, Dembo
and Zeitouni \cite{MR2571413}, Theorem 3.7.1,
\[
\limsup_{m \to\infty} \frac{m}{b_m^2} \log\P_0 \Biggl
\{\sum_{t=1}^m X_{t,i}
X_{t,j} > b_m \Biggr\} \le -\frac{1}2
\]
for any sequence $(b_m)$ such that $\sqrt{m}= \oo(b_m)$ and $b_m/m\to0$.
We choose $b_m = \sqrt{5 m \log n}$ and use the union bound, to get
\[
\P_0\{Y_{\max} > \sqrt{5 m \log n}\} \le{n \choose2}
\P_0 \Biggl\{\sum_{t=1}^m
X_{t,i} X_{t,j} > b_m \Biggr\} \to0 .
\]
Under the alternative hypothesis when $S \subset[n]$ is anomalous,
pick $i \neq j$ in $S$ such that $X_{t,i} X_{t,j}, t= 1, \dots, m$ are
i.i.d. with
mean larger than $\Rho$ and variance smaller than $2$, and by
Chebyshev's inequality,
\[
\sum_{t=1}^m X_{t,i}
X_{t,j} = m \Rho+ \OO_P(\sqrt{m}) .
\]
From this, the result follows immediately.
%
\end{pf*}

%
%
%
%
%

\section{Clique model}
\label{sec:clique}

In this section, we discuss the implications of the general results of
the previous sections for the clique model.
We derive a lower bound based on Theorem~\ref{thm:lower} in various
ranges of
the parameters and compare it with the performance bounds for the
squared-sum test and the scan statistics-based test considered in
Section~\ref{sec:tests}.
We also propose a goodness-of-fit test for the case where $\rho\to1$,
and consider two alternative tests that take computational
considerations into account.
A~digest is provided at the end of the section.

\subsection{Lower bounds for the clique model}
\label{sec:clique-lower}

In order to apply Theorem~\ref{thm:lower}, note that in the clique model,
$Z$ has hypergeometric distribution with parameters $(n,k,k)$,
which is stochastically bounded by the binomial distribution with
parameters $(k, p)$, where $p := k/(n-k)$.
In particular, for all $\xi\ge0$,
%
\begin{equation}
\label{mgf} \E\exp(\xi Z) \le\bigl(1 -p + p \mathrm{e}^{\xi}
\bigr)^k,
\end{equation}
by Bennett's inequality, for all $t >1$,
%
\begin{equation}
\label{ben} \P (Z \ge t k p ) \leq\exp \biggl(- \frac{k p}{1 -p} H(t)
\biggr),
\end{equation}
where $H(t) := t (\log t -1) + 1$.
Note that $H(t) \sim t^2/2$ when $t \to0$ and $H(t) \sim t \log t$
when $t \to\infty$.

We distinguish various regimes of the parameters in which the
minimax risk and the optimal test behave differently.
Note that Theorem~\ref{thm:lower} implies that reliable detection is impossible
(i.e., $R^* \to1$)
whenever one can choose $a$ such that $a\to\infty$ and
either
$\limsup\E [\cosh^m(\xi_a Z) ] \le1$ or
$\limsup\E\exp(m \nu_a Z) \le1$. This is what we do in all cases
listed below.

\noindent\textit{Case} 1: \textit{large $k$.}
Suppose that $k$ is so large and $\rho$ is so small that
%
\begin{equation}
\label{C1} \frac{k}{n} \to0 \quad \mbox{and}\quad  \frac{k^2}{n} \to\infty
\quad \mbox{and}\quad  \rho\sqrt{m} \frac{k^2}{n} \to0.
\end{equation}
We first note that these conditions imply that $\rho\to0$.
Let $\zeta= \rho\sqrt{m} k^2/n$ and choose $a,b \to\infty$ such
that $a^2 b \zeta\to0$. When $Z \le b k^2/n$, we use the fact that
$\rho a^2 b Z \to0$ and
$\cosh(x) \le1 + x^2$ when $x \in(0,1)$ to get that for all
sufficiently large $n$,
\[
\cosh^m(\xi_a Z) \le \bigl(1 + (\xi_a
Z)^2 \bigr)^m \leq\exp\bigl(m \xi_a^2
Z^2\bigr) 
= 1 + \oo(1), 
\]
since $\sqrt{m} \xi_a Z \le\sqrt{m} \frac{\rho a^2}{1-\rho^2}
\frac{b k^2}n = \OO(\zeta a^2 b) = \oo(1)$.

We now show that, if, in addition to \eqref{C1}, we have either $\rho
m \to0$ or $\rho^2mk \to0$, then
\[
\E \bigl[\cosh^m(\xi_a Z) \IND{Z > b k^2/n}
\bigr] = \oo(1).
\]
This implies that reliable detection is impossible
in this range of the parameters.

\noindent\textit{Case} 1(a).
In addition to \eqref{C1}, assume
%
\begin{equation}
\label{C1a} \rho m \to0 . 
\end{equation}
We choose $a$ such that $a^2 \rho m \to0$.
We use the bound $\cosh(x) \le\exp(x)$ and \eqref{ben}, to get
\[
\E \bigl[\cosh^m(\xi_a Z) \IND{Z > b k^2/n}
\bigr] \le\sum_{z > b k^2/n} \exp \biggl(m \xi_a
z - \frac{k p}{1 -p} H \biggl(\frac{z}{kp} \biggr) \biggr),
\]
where $p = k/(n-k)$.
We have $m \xi_a \sim m \rho a^2 = \oo(1)$ and $\frac{k p}{1 -p} H
(\frac{z}{kp} ) \sim z \log (\frac{z}{k^2/n} ) \ge
z \log(b)$\vspace*{-4pt} uniformly over $z > b k^2/n$.
Hence, eventually,
\begin{eqnarray*}
\E \bigl[\cosh^m(\xi_a Z) \IND{Z > b k^2/n}
\bigr] &\le&\sum_{z > b k^2/n} \exp \biggl(- \frac{1}2
z \log(b) \biggr)
\\
& \sim&\exp \biggl(- \frac{1}2 b \bigl(k^2/n\bigr) \log(b)
\biggr) = \oo(1) .
\end{eqnarray*}

\noindent\textit{Case} 1(b).
In addition to \eqref{C1}, assume
%
\begin{equation}
\label{C1b} \rho^2 m k \to0 . 
\end{equation}
We may assume that $k \leq m$ for otherwise $\rho m \to0$, which we
already covered. We
choose $a$ such that $a^2 \rho\sqrt{mk} \to0$ -- which implies in
particular that $a^2 \rho k \to0$. We use the bounds $\cosh(x) \le1
+ x^2$ for $x \in[0,1]$, the fact that $Z \le k$ -- since $Z = |S \cap
S'|$ with $|S| = |S'| = k$ -- and the fact that $a^2 \rho k \to0$, to get
\[
\cosh^m(\xi_a Z) \le \bigl(1 + (\xi_a
Z)^2 \bigr)^m \le\exp\bigl(m \xi_a^2
Z^2\bigr) \le\exp\bigl(2 a^4 \rho^2 m
Z^2\bigr) \le\exp\bigl(2 a^4 \rho^2 m k Z
\bigr),
\]
eventually. Since $a^4 \rho^2 m k = \oo(1)$, we apply \eqref{ben} and
proceed exactly as before, reaching the same conclusion.

\noindent\textit{Case} 2: \textit{small $k$}, \textit{moderate $m$}. Suppose
%
\begin{equation}
\label{C2} \frac{k^2}{n} \to0 \quad \mbox{and}\quad  \rho\sqrt{\frac{m k}{\log
(n/k^2)}}
\to0 \quad \mbox{and}\quad  \frac{k \log(n/k^2)}{m} \to0.
\end{equation}
Let $\zeta= \rho\sqrt{m k/\log(n/k^2)}$ and choose $a \to\infty$
such that $a \zeta\to0$ and $a^2 \rho k \to0$. The latter is
possible because \eqref{C2} implies that $\rho k \to0$. Then, as in
Case 1(b),
\[
\cosh^m(\xi_a Z) \le\exp\bigl(2 a^4
\rho^2 m k Z\bigr) .
\]
We then use \eqref{mgf} to get
\begin{eqnarray*}
\E\cosh^m(\xi_a Z) &\le& \biggl(1 + \frac{k}{n-k}
\mathrm{e}^{2 a^4 m \rho
^2 k} \biggr)^k
\\
&\le& \exp \biggl(2 \frac{k^2}n \mathrm{e}^{2 a^4 m \rho^2 k} \biggr)
\\
&=& \exp \bigl(2 \exp \bigl(2 a^4 m \rho^2 k - \log
\bigl(n/k^2\bigr) \bigr) \bigr)
\\
&=& 1 + \oo(1)
\end{eqnarray*}
and again, reliable detection is impossible by Theorem~\ref{thm:lower}.

\noindent\textit{Case} 3: \textit{small $k$}, \textit{small $m$}. Suppose
%
\begin{equation}
\label{C3} \frac{k^2}{n} \to0 \quad \mbox{and}\quad  \rho\frac{m}{\log(n/k^2)} \to0
\quad \mbox{and}\quad  \limsup\rho< 1 .
\end{equation}
Hence, there is some $\eps> 0$ fixed such that $\rho< 1 - \eps$.
Let $\zeta= \rho m/\log(n/k^2)$ and choose $a \to\infty$ such that
$a^2 \zeta\to0$. We use the fact that $\cosh(x) \le\exp(x)$, and
use the same bound on the moment generating function of $Z$, to get (eventually)
\begin{eqnarray*}
\E\cosh^m(\xi_a Z) &\le& \E\exp(m \xi_a Z)
\\
&\le& \biggl(1 + \frac{k}{n-k} \mathrm{e}^{m a^2 \rho/\eps} \biggr)^k
\\
&=& \exp \bigl(2 \exp \bigl(a^2 m \rho/\eps- \log
\bigl(n/k^2\bigr) \bigr) \bigr)
\\
&=& 1 + \oo(1), 
\end{eqnarray*}
implying that reliable detection is impossible.

\noindent\textit{Case} 4: \textit{very large $\rho$}.
Suppose
%
\begin{equation}
\label{C4} (1-\rho)^{1/2} \biggl(\frac{n}{k^2}
\biggr)^{1/m} \to\infty\quad \mbox{and}\quad  \rho\to1.
\end{equation}
This is the only situation where we bound
\begin{eqnarray*}
\E\exp(m \nu_a Z) &\le& \biggl(1 + \frac{k}{n-k}
\mathrm{e}^{m \nu_a} \biggr)^k
\\
&\le&\exp \biggl(2 \frac{k^2}n \mathrm{e}^{m \nu_a} \biggr)
\\
&\le&\exp \bigl(2 \exp \bigl(m \bigl(a^2 - \zeta\bigr) \bigr) \bigr)
,
\end{eqnarray*}
where $\zeta:= \frac{1}2 \log(1-\rho^2) + \frac{1}m \log(n/k^2) \to
\infty$ by \eqref{C4}.
Therefore, it suffices to choose $a \to\infty$ such that $a^2 =
\oo(\zeta)$, to have the last expression on the right-hand side tend to
one.

The discussion of these various regimes leads to the following.

\begin{cor} \label{cor:clique-lower}
In the clique model, under
either \eqref{C1} with \eqref{C1a} or \eqref{C1b}, \eqref{C2},
\eqref{C3}, or \eqref{C4}, $R^* \to1$.
\end{cor}

We will see below that \eqref{C1}, combined with either \eqref{C1a}
or \eqref{C1b}, is tight up to logarithmic factors.
This is also the case of \eqref{C2} and \eqref{C3}, unless
$k^2/n \to0$ as a negative power of $n$.
Note also that the result is silent in
the regime when $k^2/n \to0$ and $\rho\sqrt{m} k^2/n \to0$.
However, it is covered by \eqref{C2} when $\log(n/k^2)/k \to0$,
and by \eqref{C3} when $\log(n/k^2)/\sqrt{m} \to0$, so again it is a
matter of logarithmic factors.
The bound \eqref{C4} is also tight up to log factors when $m = \oo(\log n)$.
We mention that the typical exposition
in the detection-of-means literature, for example in Donoho and Jin
\cite{dj04},
avoids the discussion of such fine details
by assuming that $k = n^{\alpha}$ for some
$\alpha\in(0,1)$.

\subsection{Localized squared-sum test}

Next, we take a closer look at the performance of the localized
squared-sum test for the clique model. In this case,
we have $|\cC| = {n \choose k}$ so $\log|\cC| \sim k \log(n/k)$.
Plugging this into \eqref{rho-local}, we see that the
local squared-sum test is asymptotically powerful when
%
\begin{equation}
\label{scan-clique} \Rho\ge A \max \biggl(\frac{ \log(n/k)}m, \sqrt{
\frac{ \log
(n/k)}{km}} \biggr) ,
\end{equation}
and the constant $A$ is large enough.
Based on this and Corollary~\ref{cor:clique-lower},
we conclude that the test is near-optimal
in regimes \eqref{C2} and \eqref{C3}, though only up to a logarithmic
factor if
$k^2/n \to0$ slower than any power of $n$.
It is also near-optimal up to a logarithmic factor in regime \eqref
{C1} when neither \eqref{C1a} nor \eqref{C1b} is satisfied.

However, we do not have such a guarantee in the regime \eqref{C1}
(with either \eqref{C1a} or \eqref{C1b}).
In this range of parameters, it is the squared-sum test that
yields an optimal performance up to a logarithmic factor.
Also, comparing Proposition~\ref{prp:squared-sum} and
Proposition~\ref{prp:local-squared-sum}, we see that the local test
dominates when
$\max (1, (k/m)^{1/2} ) k^{3/2}/n$ tends to zero faster than
$1/\log(n/k)$.


As we mentioned earlier, computing the scan statistic $Y_\mathrm
{scan}$, or even
approximating with enough precision, seems to be fundamentally hard.
In fact, we conjecture that computing \textsl{any} test with a
near-optimal performance is fundamentally hard in some range of the parameters.
We see this is as a challenging and important research problem.

We make some progress in this direction in two ways.
In Section~\ref{sec:maxsum}, we suggest a test that has good
performance and
that is efficiently computable if $m$ is only logarithmic in $n$.
In Section~\ref{sec:SDP}, inspired by recent work of Berthet and
Rigollet \cite{berthet}, we consider a convex relaxation of the
problem following d'Aspremont, El Ghaoui, Jordan and Lanckriet \cite
{aspremont}.

\subsection{The case of \texorpdfstring{$\rho$}{rho} constant}
\label{sec:constantrho}

Assume that $\rho< 1$ is a constant, independent of $n$.
From our previous work \cite{ArBuLu11} in the case of $m=1$, we know
that $R^* \to1$ unless $k^2/n \to\infty$, in which case the
squared-sum test is asymptotically powerful.
Now we learn from Corollary~\ref{cor:clique-lower} \eqref{C3} that
$R^* \to1$
when $k^2/n \to0$ and $m = \oo(\log(n/k^2))$.
Hence, in the case of $\rho$ constant and $n/k^2$ a positive power of
$n$, a sample size $m$ sub-logarithmic in the dimension $n$ is not
enough for reliable
detection, and is qualitatively on par with the case of $m=1$.

The situation changes dramatically when the sample size $m$
becomes at least logarithmic in the dimension $n$. Indeed,
even for $k=2$, both the localized squared-sum test and the
maximum correlation test have a vanishing risk for any constant
value of $\rho$ when $\log(n)/m \to0$.
This reveals an interesting ``phase transition'' occurring when the
sample size is
about logarithmic in the dimension.

\subsection{The case of \texorpdfstring{$\rho$}{rho} tending to 1}
\label{sec:rho1}

The regime in \eqref{C4} does not have a match in either the
squared-sum test or the localized squared-sum test.
It is instead met by a goodness-of-fit test which is a variant of test
proposed and analyzed in our previous work \cite{ArBuLu11} in the same
regime with $m=1$, although the construction here is slightly different.

Here we assume that $\rho\to1$ so fast that
%
\begin{eqnarray}
\label{gof} &\displaystyle \frac{k}{m \log n} \to\infty\quad \mbox{and}\quad  \displaystyle \frac{k^2}n \to0
\quad \mbox {and}\quad  m \le n \quad \mbox{and}&\nonumber\\[-8pt]\\[-8pt]
  &(1-\rho)^{1/2} \biggl(
\displaystyle \frac{n}{k^2} \log n \biggr)^{1/m} \sqrt{m} \to0.&\nonumber
\end{eqnarray}
%
Let $\zeta$ denote the last term tending to zero in \eqref{gof}, and
choose $a \to\infty$ such that $a \max(\zeta,\allowbreak  1/\log n) \to0$.

The test we propose is based on the idea that the variables that are
positively correlated are closer together than the other variables that
are independent of each other.
Take $\eta\to0$ such that
%
\begin{equation}
\label{eta} \eta= \max \bigl( \bigl((m+1) \log(n)/n\bigr)^{1/m}, 2 a
(1-\rho)^{1/2} \sqrt {m} \bigr) ,
\end{equation}
which is possible by \eqref{gof}.
Let $w_1, \dots, w_\ell\in\bbR^m$ be an $\eta$-covering of $B(0, a
\sqrt{m})$ with $\ell\le(a \sqrt{m}/\eta)^{m}$, and let $R_s =
B(w_s, 2 \eta)$.
We count the number of data points in each $R_s$, yielding $B_s = \#\{
i\dvt  X_i \in R_s\}$, and consider the test that rejects for large values
of $\max\{B_s\dvt s=1,\dots,\ell\}$.

Under the null hypothesis, $B_s \sim\operatorname{Bin}(n, p_s)$ where
\[
p_s := (2 \pi)^{-m/2} \int_{R_s}
\mathrm{e}^{-\|x\|^2/2} \,\mathrm{d}x \le p := \eta^m.
\]
A simple combination of Bernstein's inequality and the union bound gives
\[
\max_s B_s \le n p + \sqrt{3 n p \log\ell},
\]
when $\log\ell= \oo(n p)$, which is the case when $m \log(a \sqrt {m}/\eta) = \oo(n \eta^{m})$ and is implied under \eqref{eta}, since
under our assumptions
\begin{eqnarray*}
m \log(a \sqrt{m}/\eta) - n \eta^{m} &\le& m \log(a \sqrt{m}) + \log n
- (m+1) \log n
\\
&=& m \bigl(\log a + \tfrac{1}2 \log m - \log n \bigr) \to-\infty.
\end{eqnarray*}

Under the alternative hypothesis where $S$ is anomalous, we use the
representation \eqref{eq:representation}, to get
\[
\|X_i - \sqrt{\rho} V\| = (1-\rho)^{1/2}
\|Y_i\| ,
\]
with
\[
\P \bigl(\|Y_i\| \le a \sqrt{m} \bigr) \ge1 - \frac{1}{a^2} ,
\]
by Markov's inequality.
Hence, by another application of Markov's inequality,
\[
\PROB \biggl\{\# \bigl\{i \in S\dvt  \|Y_i\| \le a \sqrt{m} \bigr\} \ge
\frac{k}2 \biggr\} \ge1 - \frac{2}{a^2} \to1 .
\]
%
Let $s$ be such that $V \in B(w_s, \eta)$, so by the triangle
inequality, the region $R_s$ contains at least $k/2$ anomalous vectors
$X_i$ with high probability.
Reasoning as in \cite{ArBuLu11}, we deduce that, in that case, $B_s
\ge np + k/2 - \OO_P(\sqrt{n p})$.

Hence, the test is asymptotically powerful when
$\frac{k}2 \ge\sqrt{4 np \log\ell}$, which is the case when
$n \eta^m m \log(a \sqrt{m}/\eta)= \oo(k^2)$.
This is seen easily because when $\eta= 2 (\log(n)/n)^{1/m}$, we have
\begin{eqnarray*}
n \eta^m m \log(a \sqrt{m}/\eta) &\le&(m+1)^2 (\log n)
\bigl(\log(a \sqrt{m}) + \log n \bigr)
\\
&=& \OO(m \log n)^2 = \oo\bigl(k^2\bigr) .
\end{eqnarray*}
Finally, when $\eta= 2 a (1-\rho)^{1/2} \sqrt{m} \ge2 (\log
(n)/n)^{1/m}$, we have
\begin{eqnarray*}
n \eta^m m \log(a \sqrt{m}/\eta) &\le& n \bigl(2 a (1-
\rho)^{1/2} \sqrt{m} \bigr)^{m} m \bigl(\log (a \sqrt{m}) +
\log n \bigr)
\\
&\le& k^2 (4 a \zeta)^{m} =\oo\bigl( k^2\bigr) .
\end{eqnarray*}
\begin{Remark*}
In \eqref{gof}, we really have in mind a setting where $\rho\to1$,
$k = n^\alpha$ with $\alpha< 1/2$ fixed, and $m = \oo(\log n)$, since
the maximum-correlation test is asymptotically powerful at $\rho$
constant when $m \ge C \log n$ and $C$ is sufficiently large.
In that case, comparing with \eqref{C4}, this goodness-of-fit test is
optimal up to a sub-logarithmic factor.
We note that another goodness-of-fit test based on $H_i := X_i/\|X_i\|$
achieves a similar bound, but with $\zeta$ defined as
\[
\zeta= (1-\rho)^{1/2} \biggl(\frac{n}{k^2} \biggr)^{1/(m-2)}
.
\]
\end{Remark*}

\subsection{Balancing detection ability and running time}
\label{sec:maxsum}

Given the often enormous size of data sets that statisticians
need to handle as an every-day practice, it is of great interest
to design computationally efficient, yet near-optimal tests.
In the case of the clique model, this is a highly non-trivial task,
because the class $\C$ has size exponential in $k$ and
computing the localized squared-sum test (or other versions of
the generalized likelihood ratio test and scan statistics) involves
a non-trivial optimization problem over all ${n \choose k}$ elements
of $\C$. In fact, often it seems that small testing risk and
computational efficiency are contradicting terms. In this section,
we show that in at least one non-trivial instance, it is possible
to design a computationally efficient (i.e., computable in time quadratic
in $n$) test that has near optimal risk.

This is the case when the sample size $m$ is (at most) logarithmic in $n$
and $k \sim n^a$ for some $a\in(0,1)$. (Recall from Section~\ref{sec:constantrho}
that this is a quite interesting range of parameters.)

To introduce a family of tests that balance detection
performance and computational complexity, let
$\ell\in\{1, \dots, m\}$ and define
\[
Y(\ell) = \max_{S: |S| = k} \max_{T: |T| =\ell} \sum
_{t\in T} \sum_{i\in S}
X_{t,i} .
\]
Since
\[
Y(\ell) = \max_{T: |T| =\ell} \sum_{i=n-k+1}^n
X_{T, (i)} , \qquad X_{T,
i} := \sum_{t \in T}
X_{t,i},
\]
where $X_{T, (1)} \le\cdots\le X_{T, (n)}$ are the ordered
$X_{T,i}$'s, the statistic $Y(\ell)$ can be computed in
$\OO({m \choose\ell} (n \log(n) k + \ell\log(m)))$
time by first sorting $(X_{T,i}\dvt  i = 1, \dots, n)$
and summing the largest $k$, for all subsets
$T$ of size $\ell$, and then maximizing over these.

For example, when $m\le\log_2 n$, then ${m \choose \ell}\le2^m \le n$
and the test may be computed in time $\OO(n^2\log n)$.
Even when $m \sim C\log n$ for some constant $C>0$, we may
choose $\ell\sim\gamma m$ such that $C\gamma\log(1/\gamma) \le1$.
In that case ${m \choose \ell}\le2^{\ell\log_2 (m/\ell)} \le n$
and again the test may be computed in time $\OO(n^2\log n)$.
The next proposition bounds the risk of the test.
%

\begin{prp} \label{prp:maxsum}
Take $\ell\leq m/7$. The test that rejects $H_0$ when
\[
Y(\ell) > a := \sqrt{2 \ell k \bigl(\ell\log(\mathrm{e} m/\ell) + k \log(\mathrm{e} n/k)\bigr)},
\]
is asymptotically powerful in the clique model when
\[
\Rho\ge3 \biggl(\frac{\log(n/k)}\ell+ \frac{\log(m/\ell
)}k \biggr) .
\]
\end{prp}

\begin{pf}
Since under the null hypothesis
$\sum_{t\in T} \sum_{i\in S} X_{t,i} \sim\cN(0, \ell k)$,
by a standard bound for the maximum of a finite set of
Gaussian variables,
\[
Y(\ell) \le\sqrt{2 \ell k \log{m \choose\ell} {n \choose k}} \le a,
\]
with probability converging to $1$.

Under the alternative hypothesis where $S$ is anomalous, we have
\[
Y(\ell) \ge\sqrt{\Rho} k (Z_{(m-\ell+1)} + \cdots+ Z_{(m)} ) ,
\]
where $Z_{(1)}\le\cdots\le Z_{(m)}$ are the ordered values of
\[
Z_t := \bigl(k + \Rho k (k-1)\bigr)^{-1/2} \sum
_{i \in S} X_{t,i} ,
\]
which are i.i.d. standard normal. Since we assume that $m \to\infty$,
$\P(Z_{(m-\ell+1)} \ge1) \to1$ when $\ell/m \le1/7 < \P
(\cN (0,1) > 1 )$. Hence, $Y(\ell) \ge\sqrt{\Rho} k \ell$ with
probability tending to one under the alternative hypothesis. Therefore,
the test is asymptotically powerful when $\sqrt{\Rho} k \ell> a$,
which follows from the assumptions when $k, \ell\to\infty$.
\end{pf}

In the regime of (\ref{C3}) with $m\sim C\log n$, we see that the test
is optimal up to a constant factor in $\rho$ when $k\sim n^a$ for
some $a<1/2$. In this range of parameters, it seems
hopeless to compute (or even approximate) the local squared-sum test.

However, when $m$ is much larger than logarithmic in $n$, this test
also requires super-polynomial computational time and therefore it is
not useful
in practice. In such cases, one may have to resort to sub-optimal tests
such as the maximum correlation test described in Section~\ref{sec:pairwise}.
It is an important and difficult challenge to find out the possibilities
and limitations of powerful detection taking computational constraints
into account.


\subsection{A convex relaxation} \label{sec:SDP}

In parallel to our work, Berthet and Rigollet \cite{berthet} study a
related problem of detecting a sparse principal component.
The setting there is the same, except for the alternative hypothesis,
where the covariance matrix is of the form $\Sigma= I + \theta v
v^\top$, with $\theta> 0$ and $v$ a unit vector with at most $k$
non-zero components.
They study a test based on $\lambda_k^{\max}(\widehat\Sigma
)$, the largest $k$-sparse eigenvalue of $\widehat\Sigma$, defined as
\[
\lambda_k^{\max}(A) = \max_{|S| = k}
\lambda^{\max}(A_S),
\]
where $A_S$ denotes the principal submatrix of $A$ indexed by $S$ and
$\lambda^{\max}(A)$ the largest eigenvalue of $A$.
This test is, in fact, intimately related to our localized squared-sum
test, as we shall see in the analysis below.
Berthet and Rigollet \cite{berthet} prove that this test -- with a
proper choice of critical value -- is near-optimal.
However, just like our localized squared-sum test, the test of Berthet
and Rigollet \cite{berthet}
is also computationally unfeasible due to the maximization over
${n \choose k}$ sets.
For computational reasons, \cite{berthet} turn to the convex relaxation
of d'Aspremont, El Ghaoui, Jordan and Lanckriet \cite{aspremont}, for
which they also establish a performance bound.

Berthet and Rigollet \cite{berthet} show that, when $n,m \to\infty$,
with high probability under the null hypothesis,
\[
\lambda_k^{\max}(\widehat\Sigma) \le1 + C \max \biggl(
\frac
{k \log(n/k)}m, \sqrt{\frac{k \log(n/k)}m} \biggr),
\]
for a universal constant $C$.
Under the alternative hypothesis where $S$ is anomalous, we have
\[
\lambda_k^{\max}(\widehat\Sigma) \ge\lambda^{\max}(
\widehat\Sigma_S) \ge\frac{1}k \sum
_{i,j \in S} \widehat\sigma_{ij} = \frac{1}{km}
Y_S ,
\]
and in \eqref{Yscan-H1}, we saw that $Y_S \ge km (1 + \Rho(k-1)) (1 -
\OO_P(1/\sqrt{m}))$.
Hence, this test is asymptotically powerful when
%
\begin{equation}
\label{lambdamax} \Rho\ge3 C\max \biggl(\frac{ \log(n/k)}m, \sqrt{
\frac{ \log
(n/k)}{km}} \biggr) ,
\end{equation}
which matches the performance of the localized squared-sum test \eqref
{scan-clique} up to a multiplicative constant.

The semidefinite relaxation of d'Aspremont, El Ghaoui, Jordan and
Lanckriet \cite{aspremont} for $\lambda_k^{\max}$ is
\[
\SDP_k(A) = \max\trace(A Z)\quad  \mbox{subject to } Z \succeq 0,
\trace(Z) = 1, |Z|_1 \le k ,
\]
where the maximum is over all positive semidefinite matrices $Z =
(Z_{st}) \in\bbR^{m \times m}$
and $|Z|_1$ denotes $\sum_{s,t} |Z_{st}|$. The quantity $\SDP_k(A)$ can be computed efficiently as it is a semidefinite program.
Berthet and Rigollet \cite{berthet} suggest to use the test statistic
$\SDP_k(\widehat\Sigma)$.
They show that, when $n,m \to\infty$, with high probability under the
null hypothesis,
\[
\SDP_k(\widehat\Sigma) \le1 + C k \max \biggl(\frac{\log
(n/k)}m,
\sqrt{\frac{\log(n/k)}m} \biggr),
\]
for a universal constant $C$, while under the alternative hypothesis,
\[
\SDP_k(\widehat\Sigma) \ge\lambda_k^{\max}(
\widehat \Sigma) \ge\bigl(1 + \Rho(k-1)\bigr) \bigl(1 - \OO_P(1/
\sqrt{m})\bigr) .
\]
Hence, the test based on the statistic $\SDP_k(\widehat\Sigma
)$ is asymptotically powerful when
%
\begin{equation}
\label{eq:reduc} \Rho\ge3 C\max \biggl(\frac{\log(n/k)}m, \sqrt{
\frac{\log
(n/k)}m} \biggr) .
\end{equation}
This rate matches \eqref{lambdamax} when $(k/m) \log(n/k) \to
\infty$, and is otherwise comparable to what the maximum correlation
test achieves \eqref{pairwise}.
Thus, the relaxed test of Berthet and Rigollet is computationally
efficient and near-optimal when the sample size is of smaller order
that $k\log(n/k)$. Note that this allows one to handle larger values
of $m$ than for the test introduced in Section~\ref{sec:maxsum} where
$m$ had to be at most a constant multiple of $\log n$.


Interestingly, Berthet and Rigollet \cite{berthet} also show that
their analysis of the relaxed test is optimal in the following sense.
If one can improve the rate given in \eqref{eq:reduc} for the relaxed
test, then one obtains an algorithm that improves by an order of
magnitude upon known results for the \emph{hidden clique problem}, see
Alon, Krivelevich and Sudakov \cite{AlKrSu99}. We refer to \cite
{berthet} for a more precise statement.

\subsection{Digest}
In this section, we briefly summarize our findings for the case of the
clique model in simplified regimes. We consider combinations of the
following settings:
\begin{itemize}
\item The sparse regime corresponds to $k \ll n^{1/2 - \epsilon}$ for
some $\epsilon> 0$. The non-sparse regime corresponds to $k \gg\sqrt{n}$.
\item The ultra-high dimensional setting corresponds to $m \ll k \log
n$ \cite{verzelen}, while the (potentially) high dimensional setting
corresponds to $m \gg k \log n$.

\end{itemize}
Our results can be summarized as follows:
\begin{itemize}
\item In the sparse high dimensional setting, detection is impossible
when $\rho\ll\sqrt{ (\frac{\log n}{m k} )}$\vspace*{-4pt} (see \eqref
{C2}). On the other hand when $\rho\gg\sqrt{ (\frac{\log n}{m
k} )}$, the localized squared-sum test is asymptotically powerful.
Furthermore in the extremely sparse case where $k$ is a constant, the
same performance is achieved by the maximum correlation test.
\item In the sparse ultra-high dimensional setting, detection is
impossible when $\rho\ll\min (\frac{\log n}m, 1 )$ (see
\eqref{C3}). This rate is matched by the localized squared-sum test,
and by the computationally efficient test of Berthet and Rigollet \cite
{berthet} based on $\SDP_k$ (see Section~\ref{sec:SDP}).
Furthermore, when $m \sim C \log n$ the test of Section~\ref{sec:maxsum} can also be computed in polynomial time (in $n$) and is
asymptotically powerful for $\rho\gg\frac{\log n}m + \frac{1}k$.
When $m \ll\log n$, detection is impossible when $(1-\rho)^{1/2}
(\frac{n}{k^2} )^{1/m} \to\infty$ (see \eqref{C4}) and the
goodness-of-fit test of Section~\ref{sec:rho1} matches that bound up
to a
sub-logarithmic factor.
\item In the non-sparse regime, the squared sum test is asymptotically
powerful for
$\rho\gg\frac{n}{k^2\sqrt{m}}$.
This rate is optimal (see \eqref{C1}) if either $m$ or $k$ is not too
large (that is either \eqref{C1a} or \eqref{C1b} is satisfied). In
case neither \eqref{C1a} nor \eqref{C1b} is satisfied, the localized
squared-sum test is asymptotically powerful.
Note that in the non-sparse case we can differentiate two regimes for
the value of $k$. If $n^{1/2} \ll k \ll n^{2/3}$, then condition
\eqref{C1b} is more demanding than the last part of condition
\eqref{C1}, while for $n^{2/3} \ll k \ll n$ it is the other way
around.
As a referee kindly pointed out, in the former case one may
further tighten the lower bound and recover missing logarithmic
factors by a more careful bounding of the moment generating function
of $Z^2$.
\end{itemize}

\section{Block model}
\label{sec:block}

Next, we discuss the consequences of our main results for the block model
which serves as a prototypical example of a ``small'' or
``parametric'' class.
We focus on the case where $\rho$ is bounded away from 1.
Specifically, we assume that $\rho\le\rho_0 < 1$, and define $C_0 =
(1 - \rho_0^2)^{-1}$.

In this model, to apply Theorem~\ref{thm:lower}, we may use the
obvious bound
$Z \le k \IND{S \cap S' \neq\emptyset}$.
Noting that $\P(S \cap S' \neq\emptyset) \le2 k/n$, we have
\[
\E\cosh^m(\xi_a Z) \le1 + \frac{2k}{n}
\cosh^m(\xi_a k) .
\]
We distinguish between two main regimes and we show that $R^*\to1$ in both
cases.

\noindent\textit{Case} 1: \textit{moderate $m$}.
Suppose
%
\begin{equation}
\label{B1} \rho k \sqrt{\frac{m}{\log(n/k)}} \to0 \quad \mbox{and}\quad
\frac{k}n \to0 \quad \mbox{and}\quad  \frac{\log(n/k)}{m} \to0 .
\end{equation}
Let $\zeta= \rho k \sqrt{m/\log(n/k)}$ and choose $a \to\infty$ such
that $a^2 \zeta\to0$ and $a^2 \rho k \to0$. The latter is possible
because \eqref{B1} implies that $\rho k \to0$. We use the bound
$\cosh(x) \le1 + x^2$ for $x \in(0,1)$, to get
\[
\cosh^m(k \xi_a) = \bigl(1 + (k \xi_a)^2
\bigr)^m \le\exp\bigl(C_0^2 a^4
\rho^2 k^2 m\bigr) ,
\]
for $n$ sufficiently large.
Then
\[
\frac{2k}{n} \exp\bigl(C_0^2 a^4
\rho^2 k^2 m\bigr) = 2 \exp\bigl(- \log(n/k) \bigl(1 -
C_0^2 a^4 \zeta^2\bigr) \bigr)
\to0 ,
\]
by our assumptions. Theorem~\ref{lower} now implies that reliable detection
is impossible in this range of the parameters.

\noindent\textit{Case} 2: \textit{small $m$}.
Suppose
%
\begin{equation}
\label{B2} \rho k \frac{m}{\log(n/k)} \to0 \quad \mbox{and} \quad \frac{k}n
\to0 .
\end{equation}
Let $\zeta= \rho k m/\log(n/k)$ and choose $a \to\infty$ such that
$a^2 \zeta\to0$. We use the bound $\cosh(x) \le\exp(x)$ to get
\[
\frac{2k}{n} \cosh^m(k \xi_a) \le2 \exp \bigl(-
\log(n/k) \bigl(1 - C_0 a^2 \zeta\bigr) \bigr) \to0 .
\]

This discussion leads to the following.

\begin{cor} \label{cor:block-lower}
In the block model, under either \eqref{B1} or \eqref{B2},
$R^* \to1$.
\end{cor}


In view of Corollary~\ref{cor:block-lower}, the squared-sum test is
near-optimal
for the
block model only when $k \asymp n$. However, the localized squared-sum
test has a much better performance.
We have $|\cC| = n$, and plugging this
into \eqref{rho-local}, we see that the localized squared-sum
test is asymptotically powerful when\vspace*{-1pt}
\[
\Rho k \ge A \max \bigl(\sqrt{(\log n)/m}, (\log n)/m \bigr) ,
\]
for a large enough constant $A > 0$.
With Corollary~\ref{cor:clique-lower}, we conclude that the test is
near-optimal
except in the case where $k/n \to0$ slower than any negative power of
$n$, where the test is optimal up to a logarithmic factor.

\section{Perfect matching model}
\label{sec:matching}

Here we work out the corollaries of our main results for the perfect
matching model. This model illustrates how one may proceed when
the model in question has a non-trivial combinatorial structure.
In order to use Theorem~\ref{thm:lower}, one needs to use the specific
properties of the class.
We focus on the case where $\rho$ is bounded away from 1.
Specifically, we assume that $\rho\le\rho_0 < 1$, and define $C_0 =
(1 - \rho_0^2)^{-1}$.

In the perfect matching model, $Z$ is distributed as the number of
fixed points in a random permutation over $\{1,\ldots,k\}$. It is well
known that\vspace*{-1pt}
\begin{equation}
\label{eq:distribZmatching} \P\{Z = z\} = \frac{1}{z!} \sum
_{s=0}^{k-z} \frac{(-1)^s}{s!} \leq\frac{1}{z!}
\biggl(\frac{1}{\mathrm{e}} + \frac{1}{(k-z+1)!} \biggr) \qquad  \forall z \in\{0,
\ldots, k\}.
\end{equation}
%
We prove that $R^* \to1$ in two main regimes with the help of
Theorem~\ref{thm:lower}. To simplify notation, we
assume that $k$ is even and recall that $n= k^2$ in this model.

\noindent\textit{Case} 1: \textit{small $m$}.
Suppose\vspace*{-1pt}
%
\begin{equation}
\label{M1} \rho\sqrt{k \max(k, m)} \to0 . 
\end{equation}
We choose $a \to\infty$ such that
$a^2 \rho\sqrt{k \max(k, m)} \to0$.
We use the bounds $\cosh(x) \le1 + x^2$ for $x \in(0,1)$ and $Z \le k$,
and the fact that $a^2 \rho k \to0$, to get, for $n$ sufficiently large,
\[
\cosh^m(\xi_a Z) \le\exp\bigl(C_0^2
a^4 \rho^2 m Z^2\bigr) \le\exp
\bigl(C_0^2 a^4 \rho^2 m k Z
\bigr).
\]
Now let $c = C_0^2 a^4 \rho^2 m k$.
Using \eqref{eq:distribZmatching}, one obtains
\begin{eqnarray*}
\E\cosh^m(\xi_a Z) & \leq& \E\exp(c Z)
\\
& \leq& \sum_{z=0}^k \frac{1}{z!}
\biggl(\frac{1}{\mathrm{e}} + \frac
{1}{(k-z+1)!} \biggr) \exp(c z)
\\
& \leq& \exp\bigl(\exp(c) - 1\bigr) + \frac{k+1}{(k/2+1)!} \exp(c k)
\\
& \leq& 1 + \oo(1),
\end{eqnarray*}
because $c \to0$ and $\log[(k/2+1)!] \sim(k/2) \log k$ as $k \to
\infty$.

\noindent\textit{Case} 2: \textit{moderate $m$}.
Suppose
%
\begin{equation}
\label{M2} \frac{\rho m}{\log(\min(k, m))} \to0. 
\end{equation}
We choose $a \to\infty$ such that $a^2 \rho m / \log(\min(k, m))
\to0$.
Using \eqref{eq:distribZmatching}, one obtains
\begin{eqnarray*}
\E\cosh^m(\xi_a Z) & \leq& \E\cosh^m(
\xi_a Z) \IND{Z < k/2} + \P\{Z \geq k/2\} \cosh ^m(
\xi_a k)
\\
& \leq& \sum_{z=0}^{k/2-1} \frac{1}{z!}
\biggl(\frac{1}{\mathrm{e}} + \frac
{1}{(k/2)!} \biggr) \cosh^m(
\xi_a z) + \P\{Z \geq k/2\} \exp(\xi_a m k)
\\
& \leq& \frac{1}{\mathrm{e}} \sum_{z=0}^{+\infty}
\frac{1}{z!} \cosh^m(\xi_a z) + \biggl(
\frac{k}{(k/2)!} + \P\{Z \geq k/2\} \biggr) \exp(\xi_a m k) .
\end{eqnarray*}
Now we take care separately of these last two terms. First, note that
\[
\frac{k}{(k/2)!} + \P\{Z \geq k/2\} \leq\frac{3 k}{(k/2)!} \le\exp\bigl((k/3)
\log k\bigr)
\]
when $k$ is large enough, and since
$\xi_a m k = \OO(a^2 \rho m k) = \oo(k \log k)$ by our choice of $a$, we obtain
\[
\biggl(\frac{k}{(k/2)!} + \P\{Z \geq k/2\} \biggr) \exp(\xi_a m k)
\to0 .
\]
For the other term, the situation is slightly more subtle. Let
$Y$ be a sum of $m$ independent Rademacher random
variables. Using the binomial
identity, it is easy to prove that
\[
\cosh^m(\xi_a z) = \E\exp(\xi_a z Y) ,
\]
and thus we have
\[
\frac{1}{\mathrm{e}} \sum_{z=0}^{+\infty}
\frac{1}{z!} \cosh^m(\xi_a z) = \E \bigl[\exp\bigl(
\exp(\xi_a Y) - 1\bigr) \bigr] .
\]
Now thanks to Hoeffding's inequality, we obtain for any $t>0$,
\[
\E \bigl[\exp\bigl(\exp(\xi_a Y) - 1 \bigr)\bigr] \leq\exp\bigl(\exp(
\xi_a t) - 1\bigr) + \exp\bigl(- t^2\bigr) \exp\bigl(\exp(
\xi_a m) - 1\bigr) .
\]
In particular with $t = m / \log m$, using that
$\xi_a m = \OO(a^2 \rho m) = \oo(\log m)$ by our choice of $a$, this shows that
\[
\E(\exp\bigl(\exp(\xi_a Y) - 1\bigr) = 1 + \oo(1).
\]

This discussion leads to the following.

\begin{cor} \label{cor:matching-lower}
Consider the class of perfect matchings on the complete bipartite graph.
Under either of \eqref{M1}, or \eqref{M2}, $R^* \to1$.
\end{cor}

It is easy to derive upper bounds for the performance of the
localized squared-sum test in this model. All we need to observe is
that $|\cC| =
k!$ and therefore $\log|\cC| \sim k \log k$ when $k \to\infty$.
Plugging this into \eqref{rho-local}, we see that the local
squared-sum test is asymptotically powerful when
\[
\Rho k \ge A \max \bigl(\sqrt{k \log(k)/m}, k \log(k)/m \bigr) .
\]
Thus ignoring logarithmic factors, the requirement is that
$\Rho\sqrt{m \min(k,m)}$ be large.
Looking at Corollary~\ref{cor:matching-lower}, the complement
of \eqref{M1} \textsl{or} \eqref{M2} corresponds (roughly) to
$\rho\sqrt{k \max(k,m)} \to\infty$
\textsl{and} $\rho m / \log\min(k, m) \to\infty$, which is the
same requirement if we ignore logarithmic factors. Thus the local
squared-sum test is near-optimal.

\section{The clique number of random geometric graphs}
\label{sec:cliquergg}

In this section we describe a, perhaps unexpected, application
of Theorem~\ref{thm:lower}. We use this theorem to derive a lower
bound for
the clique number of random geometric graphs on high-dimensional spheres.

To describe the problem, let $p\in(0,1)$ and let
$Z_1,\ldots,Z_n$ be independent
random vectors, uniformly distributed on the unit sphere
$S_{d-1}=\{x\in\R^d\dvt  \|x\|=1\}$.
A random geometric graph $G(n,d,p)$ is defined by vertex set
$V=\{1,\ldots,n\}$ and vertex $i$ and vertex $j$ are connected
by an edge if and only if $(Z_i,Z_j)\ge t_{p,d}$ where the threshold
value $t_{p,d}$ is such that
\[
\PROB\bigl\{ (Z_1,Z_2) \ge t_{p,d}\bigr\} = p
\]
(i.e., the probability that an edge is present equals $p$).
The \textsl{clique number} $\omega(n,d,p)$ is the size of the largest
clique of $G(n,d,p)$ (i.e., the largest fully connected subset
of vertices). In Devroye, Gy\"orgy, Lugosi and Udina \cite
{DeGyLuUd11} the behavior of the random variable
$\omega(n,d,p)$ is studied for fixed values of $p$ when $n$ is
large and $d=d_n$ grows as a function of $n$. The rate of growth of
$\omega(n,d,p)$
is shown to depend in a crucial way of how fast $d_n$ increases with $n$.
Specifically, the following results are established (and hold with
probability converging to $1$ as $n\to\infty$):
\begin{eqnarray*}
 d_n = \OO(1) \quad & \Rightarrow&\quad
\omega(n,d,p) = \Omega(n),
\\
d_n \to\infty\quad & \Rightarrow&\quad \omega(n,d,p) = \oo(n),
\\
d_n = \oo(\log n) \quad & \Rightarrow&\quad \EXP\omega(n,d,p) = \Omega
\bigl(n^{1-\epsilon}\bigr) \quad \mbox{for all }\epsilon>0,
\\
d_n \ge9 \log^2 n \quad & \Rightarrow&\quad \omega(n,d,p) = \OO
\bigl(\log^3 n\bigr),
\\
d_n/\log^3 n \to\infty\quad  & \Rightarrow&\quad \omega(n,d,p) =
\bigl(2+\oo(1)\bigr)\log_p n ,
\end{eqnarray*} %
where $a_n = \Omega(b_n)$ means that $b_n = \OO(a_n)$.
We see that the clique number behaves in drastically different ways
between $d_n=\oo(\log n)$
--  when $\omega(n,d,p)$ grows almost linearly -- and $d_n \sim\log^2
n$~--
when $\omega(n,d,p)$ has a poly-logarithmic growth at most.

The above-mentioned results leave open the question of where
exactly the ``phase transition'' occurs, and whether
the upper bound in the regime $d_n \sim\log^2 n$ is sharp. In this
section we are able to answer both of these
questions. Below we establish a general lower bound for the
clique number which implies that, perhaps surprisingly,
the phase transition occurs around $\log^2 n$ and that the upper bounds
above cannot be improved in an essential way. We show that
the median of the clique number $\omega(n,d,p)$ is bounded
from below by $\exp(\kappa\log^2 n/d)$ where $\kappa$ is a positive constant
that depends on $p$ only. This implies, for example, that
if $d\sim c\log n$ for some $c>0$, then $\omega(n,d,p)$
grows as a positive power of $n$. On the other hand, even when
$d\sim\log^{2-\epsilon} n$ for any fixed $\epsilon>0$, then
$\omega(n,d,p)$ is much larger than any power of $\log n$.
For the sake of simplicity, we only state the result for the case of
$p=1/2$. The argument is identical for other values of $p$.
%
\begin{thm} \label{thm:clique}
There exist universal constants $c_1,c_2,c_3,c_4>0$ such that
for all $n,d$ such that $d\ge c_1\log(c_2n)$, the median
of the clique number $\omega(n,d,1/2)$ satisfies
\[
\med\bigl(\omega(n,d,1/2)\bigr) \ge c_3 \exp \biggl(
\frac{c_4\log^2  (c_2n )}{d} \biggr) .
\]
One may take $c_1=7/16$, $c_2=16\log2$, $c_3= 1/16$, and $c_4=49/5120$.
In particular,
\[
d \le c_4 \log^{2-\eps} n \quad \mbox{implies} \quad \med\bigl(\omega
(n,d,1/2)\bigr) = \Omega\bigl(\exp\bigl( \log^\eps n \bigr)\bigr) .
\]
\end{thm}
\begin{pf}
The basic idea of the proof is to define a test that works well
whenever the median clique number is small. But then the lower bound of
Theorem~\ref{thm:lower} implies that the clique number cannot be small.

Let $\omega_0 = \med(\omega(n,d,1/2))$ for short.
Consider the clique model with $m = d$, all nonzero correlations equal
to $\rho$ and $k=16\omega_0$. For $i=1,\ldots,n$, let
$X^{(i)}=(X_{i,1},\ldots,X_{i,d})\in\R^d$,
and define the random geometric graph $G$ on the normalized vectors
$Z_i= X^{(i)}/\|X^{(i)}\|$, connecting points $Z_i$ and $Z_j$ whenever
$(Z_i,Z_j) \ge0$. The test statistic we consider is the clique
number of the resulting graph, denoted by $\omega$. (This test was
suggested and analyzed in \cite{DeGyLuUd11}. Here we combine their
analysis
with Theorem~\ref{thm:lower} to derive a lower bound for the median
clique number.)

Under the null hypothesis (when $\rho= 0$), the $Z_i$'s are i.i.d.
uniform on the sphere $S_{d-1}$ implying that $G \sim G(n,d,1/2)$ and,
consequently, $\omega\sim\omega(n,d,1/2)$.
Devroye, Gy\"orgy, Lugosi and Udina \cite{DeGyLuUd11} show that,
under the alternative hypothesis, with probability at least $7/8$, the
graph contains a clique of size $k$ whenever
%
\begin{equation}
\label{clique1} \binom{k} {2} < (1/8)\mathrm{e}^{d\rho^2/10}.
\end{equation}
When this is the case, the test that accepts the null hypothesis
when $\omega< k$
has a probability of type II error bounded by $1/8$.
To bound the probability of type I error of this test, we first prove
that $\E_0 \omega< 2 \omega_0$ for any $d$ and $n$ sufficiently
large. We start with
\[
\E_0 \omega\ge2 \omega_0 \quad \Leftrightarrow\quad
\E_0 \omega- \omega_0 \ge\tfrac{1}2
\E_0 \omega\quad \Rightarrow\quad \tfrac{1}2 \E_0 \omega\le
\EXP _0 \omega- \omega_0 \le\sqrt{\operatorname{var}(
\omega)},
\]
where in the last step we used the well-known fact that the
difference between the mean and the median of any random variable
is bounded by its standard deviation.
Now observe that $\omega$,\vadjust{\goodbreak} as a function of the independent random
variables $Z_1,\ldots,Z_n$, is a configuration function in the sense
of Talagrand \cite{Tal96b} which implies that $\operatorname{var}(\omega) \le
\EXP_0
\omega$,
see \cite{BoLuMa13}, Corollary~3.8.
We arrive at
\[
\E_0 \omega\ge2 \omega_0 \quad \Rightarrow\quad \tfrac{1}2
\E_0 \omega\le \sqrt{\EXP_0 \omega}\quad  \Leftrightarrow\quad
\E_0 \omega\le4.
\]
However, it is a simple matter to show that $\EXP_0 \omega>4$
for all $d$ if $n$ is sufficiently large.
(To see this it suffices to show that the probability that $5$ random
points form a clique is bounded away from zero. This follows from the
arguments of \cite{DeGyLuUd11}.)
We then bound the probability of type I error as follows
\[
\PROB_0\{ \omega\ge k\} = \PROB_0\{ \omega\ge16
\omega_0\} \le \PROB_0\{ \omega\ge8\EXP_0
\omega\} \le\tfrac{1}{8},
\]
where we used Markov's inequality in the last line.

Combining the bounds on the probabilities of type I and type II errors,
we conclude that $R^* \le1/4$.
Put it another way,
\[
R^* >1/4\quad  \Rightarrow\quad \binom{16\omega_0} {2}
\ge(1/8)\mathrm{e}^{d\rho^2/10}.
\]
Now, by Theorem~\ref{thm:lower}, we see that
\[
(16\omega_0)^2 < 4(\ln2) n\mathrm{e}^{-16\rho d/7}
\quad \Rightarrow \quad R^* >1/4.
\]
We conclude that, for any $\rho\in(0,1)$,
\[
(16\omega_0)^2 < 4(\ln2) n\mathrm{e}^{-16\rho d/7}
\quad \Rightarrow\quad (16\omega _0)^2 \ge(1/4)\mathrm{e}^{d\rho^2/10}.
\]
Therefore, if $\rho$ is such that
$4(\ln2) n\mathrm{e}^{-16\rho d/7} > (1/4)\mathrm{e}^{d\rho^2/10}$, then
$(16\omega_0)^2 \ge(1/4)\mathrm{e}^{d\rho^2/10}$.
Choosing $\rho=(7/(16d))\log((16\log2)n)$ -- which is possible since
$d\ge(7/16)\log((16\log2)n)$~-- clearly
satisfies the required inequality and this choice gives rise to
the announced lower bound.
\end{pf}

\section{Discussion}
\label{sec:discussion}

We close this paper by discussing some open problems and directions
of further research.

\emph{Sharper bounds}.
The cornerstone of our analysis is the lower bound stated in
Theorem~\ref{thm:lower}. It is powerful enough that we can deduce useful
bounds in many different models, which are seen to be optimal up to
constant or logarithmic factors. While a considerable effort
has been devoted in the related
detection-of-means problem for finding the right constants, one
wonders if it is possible to obtain results that fine here, at least
in some regimes. One possible avenue is via the truncated
second moment approach, which underlies the lower bounds
in Ingster \cite{Ingster99}, Donoho and Jin \cite{dj04},
Hall and Jin \cite{hj09}, Butucea and Ingster \cite{ingster}.
The computations are rather
daunting in the setup of this paper and we decided not to take this
route. Note that the second moment approach (without truncation) has
limited applicability, though it is a little more useful here than it
is in the case where $m=1$.

\emph{Comparison with the detection-of-means setting}.
Our results reveal that the dependence on the sample size is different here.
In the detection-of-means setting, one reduces by sufficiency to the
case where $m=1$ by simply averaging the multiple observations and
working with $\bar{X}_1, \dots, \bar{X}_n$, where $\bar{X}_i =
\frac{1}m \sum_t X_{t,i}$. Therefore, if initially $X_{t,i} \sim\cN
(\mu, 1)$ when $i$ is anomalous, we now have $\bar{X}_i \sim\cN(\mu
, 1/m)$.
Therefore, we reduce the problem to where $m=1$ and $\mu$ is replaced
by $\sqrt{m} \mu$.
From this, we know that reliable detection is possible if either
\[
\frac{k^2}n \to\infty\quad \mbox{and}\quad  \mu^2 m \frac{k^2}n
\to\infty,
\]
or
\[
\frac{k^2}n \to0 \quad \mbox{and}\quad  \mu^2 m/\log(n) \ge C ,
\]
where $C$ is a large enough constant.
In our previous work, we argued that, at least when $\rho$ is bounded
away from 1, the parameter $\rho$ in the correlation-detection problem
played a similar role as $\mu^2$ in the detection-of-means setting.
The case when the sample size $m \to\infty$ is, however, quite different
both in the ``dense'' regime $\rho\sqrt{m} \gg n/k^2$, and in the
``intermediary'' regime $\rho\sqrt{m} \succ\sqrt{\log(n/k)/k}$.
We also note that this regime does seem to have an equivalent in the
detection-of-means setting.

\emph{Computational considerations}.
The localized squared-sum test, although near-optimal in some regimes,
is not computationally tractable in the clique model.
Following the footsteps of Berthet and Rigollet \cite{berthet}, we
find that the convex relaxation of
d'Aspremont, El Ghaoui, Jordan and Lanckriet \cite{aspremont} leads
to a computationally-tractable test of comparable performance in the
``sparsest'' regime where $m = \OO(k \log(n))$.
The best performance achievable by tests that can be computed in
polynomial time is yet unknown and remains an intriguing open problem.
The work of Berthet and Rigollet \cite
{berthet,berthet2013computational} (on detecting a sparse first
principal component) provides a possible route via reduction to the
Planted Clique Problem.

\emph{General correlations}.
More generally, the problem of detecting correlations of arbitrary
sign~-- not just positive correlations like we do here -- remains
open. Even though one can design natural tests akin to our squared-sum and
local squared-sum tests for that situation, the challenge is in
deriving tight lower bounds. We mention that our approach to obtaining
a lower bound
in Section~\ref{sec:lower} does not apply here, since the
representation \eqref
{eq:representation} is not valid when the correlations may be negative.

\section*{Acknowledgements}

We would like to thank the anonymous referees for helpful comments and
suggestions.
We also thank Quentin Berthet and Philippe Rigollet for shedding some
light on their results at the \emph{Nonparametric and High-dimensional
Statistics} conference, held in December of 2012, in Luminy, France.
EAC was partially supported by NSF grant DMS-11-20888 and ONR grant
N00014-09-1-0258. GL was supported by the Spanish Ministry of Science
and Technology grant MTM2012-37195.




\printhistory

\end{document}